\documentclass{amsart}
\usepackage{amsthm}
\usepackage{amstext}
\usepackage{amssymb}
\usepackage{mathrsfs}
\usepackage{mathabx}
\usepackage[english]{babel}
\usepackage{enumerate}
\usepackage{xfrac}
\usepackage{mathabx}
\usepackage{mathtools}
\usepackage{xcolor}
\usepackage[all]{xy}

\usepackage[T1]{fontenc}

\theoremstyle{plain}
\newtheorem{theorem}{Theorem}[section]
\newtheorem{proposition}[theorem]{Proposition}
\newtheorem{corollary}[theorem]{Corollary}
\newtheorem{lemma}[theorem]{Lemma}

\theoremstyle{definition}

\newtheorem{example}[theorem]{Example}
\newtheorem{remark}[theorem]{Remark}

\theoremstyle{remark}

\numberwithin{equation}{section}

\newcommand{\R}{\mathbb R}
\newcommand{\C}{\mathbb C}

\newcommand{\Id}{\mathrm{Id}}
\newcommand{\GL}{\operatorname{GL}}
\newcommand{\SO}{\operatorname{SO}}
\newcommand{\Spin}{\operatorname{Spin}}

\newcommand{\spec}{\operatorname{Spec}}
\newcommand{\Hom}{\operatorname{Hom}}

\newcommand{\End}{\operatorname{End}}
\newcommand{\Dom}{\operatorname{Dom}}
\newcommand{\sgn}{\operatorname{sgn}}
\newcommand{\abs}{\operatorname{abs}}

\newcommand{\dd}{\mathrm{d}}

\newcommand{\IC}{\mathbb{C}}
\newcommand{\IR}{\mathbb{R}}

\DeclareMathOperator{\mabs}{abs}

\DeclareRobustCommand\abs[1]{\left\vert #1 \right\vert}

\def\rm{\mathrm}
\def\bf{\mathbf}

\newcommand{\ICC}{C^{\infty}}

\newcommand{\loc}{\mathrm{loc}}

\newcommand{\ILL}{\mathscr{L}}

\newcommand{\IHH}{\mathscr{H}}

\newcommand{\IFF}{\mathscr{F}}

\newcommand{\dom}{\mathrm{Dom}}

\newcommand{\IN}{\mathbb{N}}

\newcommand{\IP}{\mathbb{P}}

\newcommand{\transport}{\slash\slash}

\newcommand{\slim}{\mathop{\rm{st}\rule[0.5ex]{1ex}{0.1ex}\mathrm{lim}}}

\newcommand\newdot{{\kern.8pt\cdot\kern.8pt}}
\def\nbull{{\raise1.5pt\hbox{\bf .}}}

\def\transport{/\hspace*{-3pt}/}

\title{Scattering Theory and Spectral Stability under a Ricci Flow for Dirac Operators }
\author{Sebastian Boldt}
\address{Mathematisches Institut\\Universit\"at Leipzig\\04081 Leipzig\\Germany}
\email{boldt@math.uni-leipzig.de}
\author{Batu G\"uneysu}
\address{Mathematisches Institut\\Humboldt-Universit\"at zu Berlin\\10099 Berlin}
\email{gueneysu@math.hu-berlin.de }

\begin{document}

\begin{abstract}
	Given a noncompact spin manifold $M$ with a fixed topological spin structure and two complete Riemannian metrics $g$ and $h$ on $M$ with bounded sectional curvatures, we prove a criterion for the existence and completeness of the wave operators $\mathscr{W}_{\pm}(D_h, D_g, I_{g,h})$ and $\mathscr{W}_{\pm}(D_h^2, D^2_g, I_{g,h})$, where $I_{g,h}$ is the canonically given unitary map between the underlying $L^2$-spaces of spinors. This criterion does not involve any injectivity radius assumptions and leads to a criterion for the stability of the absolutely continuous spectrum of a Dirac operator and its square under a Ricci flow.   
\end{abstract}

\subjclass[2010]{Primary 35P25, Secondary 53C27, 58J65}

\maketitle

\tableofcontents

\section{Introduction}

Assume that $M$ is a noncompact spin manifold with a fixed topological spin structure and that $g$ and $h$ are complete Riemannian metrics  on $M$ with the induced Dirac operators $D_g$ and $D_h$, acting in their respective Hilbert space of square integrable spinors $\Gamma_{L^2}(M,\Sigma_g(M))$ and $\Gamma_{L^2}(M,\Sigma_h(M))$. This paper deals with the following question\footnote{The basic concepts of scattering theory needed in this paper are summarized in section \ref{saab}.}:\vspace{1mm}

\emph{Which smallness assumptions on $g$, $h$ and the deviation of $g$ from $h$ guarantee the existence and completeness of the wave operators $\mathscr{W}_{\pm}(D_h,D_g;I_{g,h})$ and $\mathscr{W}_{\pm}(D_h^2,D_g^2;I_{g,h})$?}\vspace{1mm}

Above, 
$$
I_{g,h}:\Gamma_{L^2}(M,\Sigma_g(M))\longrightarrow \Gamma_{L^2}(M,\Sigma_h(M)
$$
denotes the canonically given map which is induced from writing $h$ as a multiplicative perturbation of $g$ (cf. section \ref{prel} for a precise definition of this identification operator). In particular, the above problem is a genuine two-Hilbert-space scattering problem. As usual in scattering theory, the existence and completeness of $\mathscr{W}_{\pm}(D_h,D_g;I_{g,h})$ (resp. $\mathscr{W}_{\pm}(D_h^2,D_g^2;I_{g,h})$) implies that the absolutely continuous spectra of $D_g$ and $D_h$ (resp. of $D_g^2$ and $D_h^2$) are equal, thus any solution of the above problem automatically provides a contribution to the spectral geometry of noncompact manifolds, although of course the existence and completeness of wave operators is a much stronger statement than the equality of the absolutely continuous spectra.\vspace{3mm}

In order to formulate our main result, given a Riemannian metric $g$ on $M$ we denote by $R_g$ its curvature tensor, with $\nabla^g$ its Levi-Civita connection, with $B_g(x,r)$ the open geodesic balls, and with $\mu_g$ the volume measure. If $g$ is complete, we further set
\begin{align*}
\Psi_g:M\longrightarrow \IR, \quad \Psi_g(x):=\big(1+\max_{y\in B_g(x,1)}|\nabla^gR_g(y)|\big)^2.
\end{align*}

Assuming now $g$ and $h$ are Riemannian metrics on $M$, define a fiberwise positive endomorphism 
$$
\mathscr{A}^g_h:TM\longrightarrow TM,\quad h(X_1,X_2)=g(\mathscr{A}^g_hX_1,X_2),
$$
which is self-adjoint with respect to $g$ and $h$. By taking the fiberwise operator norm\footnote{Here it is irrelevant whether the fiberwise operator norm is taken w.r.t.\ $g$ or w.r.t.\ $h$, as by the self-adjointness $\mathscr{A}^g_h$, this number is just the largest eigenvalue on the given fiber} of a certain normalization of $\mathscr A^g_h$ we get a function
$$
\delta_{g,h} : M\longrightarrow \R
$$
that measures the deviation of the metrics at a zeroth order level (cf.\ section \ref{prel} for a precise definition). In order to measure a first order deviation of the metrics, define the function
$$
\omega_{g,h}:M\longrightarrow \R\,,\quad \omega_{g,h}(x):= |\nabla^h-\nabla^g|_g(x)\,.
$$

Using these functions, we finally define the weight functions
\begin{align*}
\Psi^1_{g,h}(x)&:=\max\big(\delta_{g,h}(x)^2,\omega_{g,h}(x)^2, \delta_{g,h}(x)\Psi_g(x)\big)\,,\\
\Psi^2_{g,h}(x)&:=\max\big(\omega_{g,h}(x), \delta_{g,h}(x)\Psi_g(x),\delta_{g,h}(x)\Psi_h(x)\big)\,.
\end{align*}
Now our main results read as follows (cf. Theorem \ref{main}):\vspace{3mm}

Assume $g$ and $h$ are complete and quasi-isometric, and there exists a constant $C$ with $|\omega_{g,h}|+|R_g|+|R_h|\leq C$.\vspace{1mm}

a) If for some $t>0$ and some (and then by quasi-isometry: both) $j\in \{g,h\}$ one has
\[
	\int_M \frac{\Psi^{(1)}_{g,h}(x)}{\mu_j(B_j(x,\sqrt{t}))} \dd \mu_j(x)<\infty\,,
\] 
then the wave-operators $\mathscr{W}_\pm(D_h,D_g,I_{g,h})$ exist and are complete and one has $\spec_{\mathrm{ac}}(D_h) = \spec_{\mathrm{ac}}(D_{g})$.\vspace{2mm}

b) If for some $t>0$ and some (and then by quasi-isometry: both) $j\in \{g,h\}$ one has
\[
	\int_M \frac{\Psi^{(2)}_{g,h}(x)}{\mu_j(B_j(x,\sqrt{t}))} \dd \mu_j(x)<\infty\,,
\] 
then the wave-operators $\mathscr{W}_\pm(D^2_h,D^2_g,I_{g,h})$ exist and are complete and one has $\spec_{\mathrm{ac}}(D_h^2) = \spec_{\mathrm{ac}}(D_{g}^2)$.\vspace{3mm}

To the best of our knowledge, this result is even philosphically entirely new, in the sense that for the first time \emph{arbitrary metric perturbations of Dirac operators} are being treated. There are only two comparable results we are aware of. One is the scattering theory for the Hodge-Laplacian on $k$-forms treated in \cite{bgm}, which, however only treats \emph{conformal perturbations} of the metrics, a situation which is much easier to handle, as then the analogue of $\mathscr{A}^g_h$ is a scalar factor which conveniently commutes with several data (moreover the results in \cite{bgm} lead to assumptions on the underlying injectivity radii). The other is a trace class result\footnote{from which one can deduce a scattering result by the invariance principle of the wave operators} for the differences of the semigroups induced by squares of generalized Dirac operators in \cite[Ch.\ 6, Sec.\ 2]{Eichhorn}, which comes with a rather long list of assumptions, one being \emph{bounded geometry} of the underlying manifold, making it difficult to apply. Some scattering theory for Dirac operators, which essentially treats perturbations by compact sets, can be found in \cite{bunke}, while the current state of the art concerning arbitrary metric perturbations of the scalar Laplace-Beltrami operator can be found in \cite{gt, hempel}. In any case it should be noted that, in contrast to the above references, the situation we treat in this paper is a genuine $2$-Hilbert space scattering problem, in the sense that not only the scalar products in the underlying Hilbert spaces are changed, but also the underlying spaces themselves differ (since, given a fixed spin structure, any two different metrics leads to different spinor bundles). This fact requires some considerable extra machinery the complexity of which is reflected ultimately by results such as Theorem \ref{thm:dirac-hpw-formula} and Theorem \ref{thm:dirac-hpw-formula2} below.\vspace{2mm}

The main strenght of our result is given by the fact that we do not have to impose and control on the underlying injectivity radii. We achieve this by using parabolic methods, rather then elliptic methods. More precisely, keeping the Belopolskii-Birmann theorem in mind (cf. section \ref{saab}), the main step is to prove that for a) the operator
$$
\mathscr{R}_{g,h,t}:= D_h\exp(-tD^2_{h}) I_{g,h} \exp(-tD^2_{g}) - \exp(-tD^2_{h}) I_{g,h} \exp(-tD^2_{g}) D_g
$$
is trace class, while for b) that the operator
$$
\mathscr{T}_{g,h,t}:= D^2_h\exp(-tD^2_{h}) I_{g,h} \exp(-tD^2_{g}) - \exp(-tD^2_{h}) I_{g,h} \exp(-tD^2_{g}) D^2_g
$$
is trace class. To achieve this, we use the machinery for metric perturbations of Riemannian spin structures by Bourguignon and Gauduchon \cite{BG92} in order to decompose these operators in a form that allows us to restrict ourselves to the derivations of Hilbert-Schmidt estimates for operators of the form $AB_j\exp(-sD^2_j)$, where $A$ is multiplication operator and $B_j$ is either $D_j$ or the Spin-Levi-Civita connection w.r.t. $j\in \{g,h\}$. The proofs of these decomposition formulae are rather technical and are the contents of Theorem \ref{thm:dirac-hpw-formula} and Theorem \ref{thm:dirac-hpw-formula2}, respectively.\\
In order to obtain Hilbert-Schmidt estimates for the operators $AB_j\exp(-sD^2_j)$, we adjust the probabilistic machinery by Driver and Thalmaier \cite{driver} to our situation: ultimately, in the spirit of the Feynman-Kac formula, it turns out that it is possible obtain path integral formulae in terms of Brownian motion for the operators $B_j\exp(-sD^2_j)$. These so called \emph{Bismut derivative formulae} involve certain stochastic processes, which reflect the underlying geometry in a very transparent way, and thus allow to obtain explicit estimates. These formulae are the content of Theorem \ref{B1} and of Theorem \ref{B2}. \vspace{2mm}

The fact that our main result does not require any control on injectivity radii makes it accessible to Ricci flow and we prove the following result in this context (cf. Theorem \ref{thm:ricci-flow}):\vspace{2mm}

Let $S>0$, $\kappa \in\R$ and let $(g_s)_{s\in[0,S]}$ be a smooth family of Riemannian metrics on $M$. Assume
	\begin{itemize}
		\item[(i)] $g_0$ is geodesically complete with $|R_0|_0\leq C <\infty$;
		\item[(ii)] $(g_s)_{s\in[0,S]}$ evolves under a Ricci type flow
		\[
			\frac{\partial}{\partial s}g_s = \kappa \mathrm{Ric}_s \text{ for all } s\in[0,S]\,;
		\]
		\item[(iii)] there exist positive constants $C_0,C_1$ such that
		\begin{align*}
			|R_s|_s \leq C_0 \quad \text{ and } \quad |\nabla^sR_s|_s\leq C_1/s \quad \text{ for all } s\in (0,S]\,.
		\end{align*}
	\end{itemize}
For every $s_0\in(0,S)$, $x\in M$, set
	\begin{align*}
			\mathcal A_{s_0}(x) &:= \sup\left\{|\mathrm{Ric}_s(v,v)|\,:\, v\in T_xM, |v|_s\leq 1, s\in[s_0,S] \right\}\,,\\
				\mathcal B_{s_0}(x) &:= \sup\{|\nabla^s_v\mathrm{Ric}_s(u,w) + \nabla^s_u\mathrm{Ric}_s(v,w)-\nabla^s_w\mathrm{Ric}_s(u,v)| \,:\, u,v,w\in T_xM,\\ &\mkern72mu|u|_s,|v|_s,|w|_s\leq 1, s\in[s_0,S] \}\,.
	\end{align*}
	
a) If for some $s_0\in (0,S)$ one has
	\[
		\int_M\! \frac{\max(\sinh\left(\tfrac n4 (S-s_0)|\kappa|\mathcal A_{s_0}(x)\right),\sinh^2\left(\tfrac n4 (S-s_0)|\kappa|\mathcal A_{s_0}(x)\right),\mathcal B^2_{s_0}(x))}{\mu_{s_0}( B_{s_0}(x,1))}\dd\mu_{s_0}(x) < \infty
	\]
	then the wave operators $\mathscr W_\pm(D_s,D_{s_0},I_{s_0,s})$ exist and are complete and one has $\spec_{\mathrm{ac}}(D_s) = \spec_{\mathrm{ac}}(D_{s_0})$ for all $s\in [s_0,S]$.\vspace{1mm}
	
	b) If for some $s_0\in (0,S)$ one has
		\begin{align*}
		\int_M\! \frac{\max\left(\sinh\left(\tfrac n4 (S-s_0)|\kappa|\mathcal A_{s_0}(x)\right),\mathcal B_{s_0}(x)\right)}{\mu_{s_0}( B_{s_0}(x,1))}\dd\mu_{s_0}(x) < \infty
	\end{align*}
	then the wave operators $\mathscr W_\pm(D^2_s,D^2_{s_0},I_{s_0,s})$ exist and are complete and one has $\spec_{\mathrm{ac}}(D^2_s) = \spec_{\mathrm{ac}}(D^2_{s_0})$ for all $s\in [s_0,S]$.
\vspace{2mm}

Note that the assumptions of this result are natural: a typical short time existence result for the Ricci flow on noncompact manifolds (see, e.g., \cite{shi}) asserts that, given any Riemannian metric $g_0$ satisfying \textrm{(i)}, there exists a solution $(g_s)_{s\in[0,S]}$ of the Ricci flow 
	\[
		\frac{\partial}{\partial s}g_s = -2 \mathrm{Ric}_s\,,\quad  s\in[0,S]\,
	\]
for some $S>0$ which satisfies \textrm{(iii)}.\vspace{3mm}

This paper is organized as follows: in section \ref{prel} we introduce the basic notions from spin geometry, allowing to introduce the elements of the Bourguignon/Gauduchon machinery. Section \ref{hpwform} deals with the proofs of the aforementioned decomposition formulae Theorem \ref{thm:dirac-hpw-formula} and Theorem \ref{thm:dirac-hpw-formula2}, and section \ref{bismut} deals with the Bismut derivative formulae Theorem \ref{B1} and Theorem \ref{B2}, as well as the resulting Hilbert-Schmidt estimates. Section \ref{mainres} is devoted to the proof of our main result Theorem \ref{main}. Finally, section \ref{ricci} is devoted to the application of our main result to the Ricci flow, that is, the proof of Theorem \ref{thm:ricci-flow}.\\
In addition, we have included an appendix which summarizes the basic concepts of stochastic analysis that are needed in the context of Bismut derivative formulae, aiming to make the paper essentially self-contained.

\vspace{5mm}
\begin{center}
	\textit{We dedicate this paper to the memory of P.\ Grabowski.}
\end{center}
\vspace{5mm}
\section{Preliminaries}\label{prel}


This section serves to fix notation, describe the set-up and recall the Bourguignon-Gauduchon theory of metric variations of the (spin-)Dirac operator \cite{BG92}. We also point the reader to \cite[Appendix A]{NM07} which contains an excellently translated and extended excerpt of \cite{BG92}\footnote{Unfortunately, there are three typos in the proof of Theorem~A.8. In the first two display formulae the expression $b^h_g(X)$ needs to be replaced by $X$ and in the third display formula $\widetilde{\nabla}^\eta_{e_i}$ has to be $\widetilde{\nabla}^\eta_{H_g^{-1/2}(e_i)}$.}.

Let $M$ be a smooth oriented noncompact manifold of dimension $n\geq 2$ and let $\GL_+(M)$ denote the $\GL_+(n)$-principal bundle over $M$ of oriented frames. Then a \emph{topological spin structure on $M$} is a double cover $\widetilde{P}\to \GL_+(M)$ by a $ \widetilde{\GL_+}(n)$-principal bundle $\widetilde{P}$ over $M$, which is equivalent to $\widetilde{\GL_+}(n)\to \GL_+(n)$ when restricted to the fibers over $M$, where  the latter map denotes the universal double cover if $n\geq 3$, and the connected double cover if $n=2$. If $M$ admits a spin structure, one calls $M$ a spin manifold. This condition is equivalent to the assumption that the second Stiefel-Whitney class $w_2(TM)\in H^2(M,\mathbb{Z}_2)$ of $TM$ vanishes.\vspace{2mm}

\emph{We assume throughout that $M$ is a spin manifold and we fix a topological spin structure on $M$.}

\vspace{2mm}

Then, given a Riemannian metric $g$ on $M$, the restriction $P_g$ of $\widetilde{P}$ to $\SO(M,g)\subseteq \GL_+(M)$ becomes a \emph{Riemannian spin structure} in the usual sense (see \cite[Ch. II, Theorem 1.4]{LM89}), where $\SO(M,g)$ denotes the $\SO(n)$-principal bundle over $M$ of oriented $g$-orthonormal frames. In other words, $P_g$ is a $\Spin(n)$-principal fibre bundle over $M$ which reduces $\SO(M,g)$ in the sense of principal fibre bundle and for which each fiber $(P_g)_x$ is a nontrivial double cover of $\SO(M,g)_x$ for each $x\in M$. In particular, we canonically get the Hermitian vector bundle $\Sigma_gM\to M $ of spinors. We denote by 
$$
\nabla^g:\Gamma_{\ICC}(M,TM)\longrightarrow \Gamma_{\ICC}(M,T^*M\otimes TM)
$$
the Levi-Civita connection of $g$ and with
$$
 \widetilde{\nabla}^g:\Gamma_{\ICC}(M,TM)\longrightarrow \Gamma_{\ICC}(M,T^*M\otimes \Sigma_g M) 
$$
its lift to the spinor bundle, where we recall that although the tensor product $T^*M\otimes \Sigma_g M$ is over $\R$, the bundle carries a canonical complex structure, where complex multiplication is given by multiplication on the second factor (and likewise for $TM\otimes \Sigma_g M$). The Clifford multiplication will be denoted with 
$$
TM\otimes \Sigma_gM\longrightarrow \Sigma_g M,\quad v\otimes\varphi\longmapsto v\underset{g}{\cdot} \varphi
$$ 
and the Dirac operator with
$$
D_g=  \sum_{i=1}^n e_i\underset{g}{\cdot} \widetilde{\nabla}^g_{e_i}:\Gamma_{C^{\infty}}(M,\Sigma_g M)\longrightarrow \Gamma_{C^{\infty}}(M,\Sigma_g M),
$$
where $(e_1,\dots,e_n)$ is a local $g$-ONB. Furthermore, 
$$
\sharp^g:T^*M\longrightarrow  TM
$$
will stand for the musical isomorphism. Finally, the volume measure is denoted with $\mu_g$, leading to the complex Hilbert space of $L^2$-spinors  $\Gamma_{L^2}(M,\Sigma_g M)$.\vspace{2mm}


For two (smooth) Riemannian metrics $g$ and $h$ on $M$ we write $g\sim h$ if $g$ and $h$ are quasi-isometric, i.e., there exists $C>0$ such that 
\begin{align}\label{eqn:quasi-isometric}
	(1/C)g\leq h\leq C g
\end{align}
in the sense of quadratic forms. \vspace{2mm}

\emph{Henceforth, we fix a pair $g,h$ of quasi-isometric geodesically complete Riemannian metrics.}\vspace{2mm}

Define a section 
$$
\text{$\mathscr A:=\mathscr A^g_h$ of $\End(TM)$ by $h(X_1,X_2): = g(\mathscr AX_1,X_2)$ for all $x\in M$, $X_1,X_2\in T_xM$.} 
$$
It follows from the symmetry of $g$ and $h$ that $\mathscr A$ is self-adjoint w.r.t.\ $g$. This in turn implies that $\mathscr A$ is also self-adjoint w.r.t.\ $h$. The positive-definiteness of $h$ (or $g$) then implies that $\mathscr A$ is positive, i.e., $\mathscr A(x)$ has only positive eigenvalues for every $x\in M$. Note that the spectral calculus of $\mathscr A(x)$ is independent of any metric: we can decompose $\mathscr A(x)$ pointwise as a linear combination of its eigenprojections, which are independent of the metrics $g$ and $h$.

Now let $P$ and $Q$ be the spin structures corresponding to the metrics $g$ respectively $h$ and the topological spin structure $\widetilde{P}$. Since $\mathscr A^{-1/2}$ is a (pointwise) isometry from $(TM,g)$ to $(TM,h)$, it lifts to an $\SO(n)$-equivariant map $b^g_h$ from $\SO(M,g)$ to $\SO(M,h)$ taking an oriented ONB $(e_1,\ldots,e_n)$ of $(T_xM,g_x)$ to the oriented ONB $(\mathscr A(x)^{-1/2}e_1,\ldots,\mathscr A(x)^{-1/2}e_n)$ of $(T_xM,h_x)$. This map in turn lifts to a $\Spin(n)$-equivariant map $\beta^g_h$ from $P$ to $Q$. By equivariance, $\beta^g_h$ now extends to a fibrewise unitary isomorphism from $\Sigma_gM$ to $\Sigma_hM$. This unitary isomorphism is moreover compatible with Clifford multiplication in the following sense:
\begin{align*}
	\beta^g_h(X\underset{g}{\cdot}\sigma) = \mathscr A^{-1/2}(X)\underset{h}{\cdot}\beta^g_h(\sigma)
\end{align*}
for all $x\in M$, $X\in T_xM$ and $\sigma\in (\Sigma_gM)_x$.

We define a bounded identification operator by
\begin{equation*}\label{eqn:definition-I}
	\begin{aligned}
				I:=I_{g,h}:\Gamma_{L^2}(M,\Sigma_gM)&\longrightarrow  \Gamma_{L^2}(M,\Sigma_hM)\\
					I_{g,h}\varphi(x)&=\beta^g_h(\varphi(x))\,,
	\end{aligned}
\end{equation*}
which is well-defined since $g\sim h$. Clearly, we have
\begin{equation}\label{eqn:I-inverse-adjoint}
	\begin{aligned}
		I_{g,h}^{-1}=I_{h,g} \quad \text{ and }  \quad I_{g,h}^*\psi(x)=\varrho_{g,h}(x)I_{h,g}\psi(x)\,,
	\end{aligned}
\end{equation}
where $0<\varrho_{g,h}\in C^{\infty}(M)$ is the Radon-Nikodym density of $\mu_h$ with respect to $\mu_g$, i.e., 
$$
\dd \mu_h=\varrho_{g,h} \dd \mu_g.
$$
The density $\varrho_{g,h}$ can be expressed in terms of $\mathscr A$ by
\begin{equation}\label{eqn:varrho-A}
		\varrho_{g,h} =\det  (\mathscr A^g_h)^{1/2}\,.
\end{equation}

We denote by $\mathscr A':=(\mathscr A^g_h(x))'$ the transpose map of $\mathscr A^g_h(x)$, i.e., $(\mathscr A^g_h(x))'\varphi = \varphi\circ \mathscr A^g_h(x)$ for all $\varphi\in T^*_xM$. Note that we have $g^*(\varphi,\psi)=h^*(\mathscr A'\varphi,\psi)$, i.e., $\mathscr A'^{1/2}$ is a (pointwise) isometry from $(T^*M,g)$ to $\mathscr (T^*M,h)$, and that the spectral calculus commutes with taking transposes, e.g., $(\mathscr A^\alpha)'=(\mathscr A')^\alpha$ for $\alpha\in\R$. For later usage, we also record the following two relations,
\begin{align}\label{eqn:spectral-calculus-of-A}
	\mathscr A^{-1/2}\sharp^g=\sharp^h\mathscr A'^{1/2}\,, & & &  f(\mathscr A)\sharp^j=\sharp^jf(\mathscr A')\,.
\end{align} 
Here, $f(\mathscr A)$ and $f(\mathscr A')$ are the spectral calculi associated with $\mathscr A$ and $\mathscr A'$, respectively, and a suitable function $f$.

In the sequel, we will work on various tensor product bundles, one of which is $T^*M\otimes \Sigma_jM$ with $j\in\{g,h\}$. To not further clutter the presentation, we will implicitly use the obvious inner product on any such bundle. For example, the bundle $T^*M\otimes \Sigma_gM$ is endowed with the inner product $g^*\otimes\gamma$, where $\gamma$ is the inner product on $\Sigma_gM$.

Let us define another bounded linear identification operator by
\begin{align*}
\widetilde{I}:=\widetilde{I}_{g,h}:\Gamma_{L^2}(M,T^*M\otimes\Sigma_gM)&\longrightarrow  \Gamma_{L^2}(M,T^*M\otimes\Sigma_hM)\\
\widetilde{I}_{g,h}\varphi(x)&=\left(\mathscr A'^{1/2}\otimes\beta^g_h\right)(\varphi(x)).
\end{align*}
Analoguosly to \eqref{eqn:I-inverse-adjoint} we have
\begin{align*}
    \widetilde{I}_{g,h}^{-1}\psi(x) &= \widetilde{I}_{h,g}\psi(x) = \left(\mathscr A'^{-1/2}\otimes\beta^h_g\right)(\psi(x))\,,\\
    \widetilde{I}_{g,h}^*\psi(x)&=\varrho_{g,h}(x)(\mathscr A'^{-1/2}\otimes\beta^h_g)(\psi(x))=\varrho_{g,h}(x)(\widetilde{I}_{h,g}\psi)(x)\,.	
\end{align*}

We introduce the following 'skewed' connections,
\begin{equation*}
	\begin{aligned}
		{^g}\nabla^h &= \mathscr A^{1/2}\circ\nabla^h\circ\mathscr A^{-1/2}:\Gamma_{C^{\infty}}(M,TM)\longrightarrow \Gamma_{C^{\infty}}(M,T^*\otimes TM)\,,\\
		{^h}\nabla^g &= \mathscr A^{-1/2}\circ\nabla^g\circ\mathscr A^{1/2}:\Gamma_{C^{\infty}}(M,TM)\longrightarrow \Gamma_{C^{\infty}}(M,T^*\otimes TM)\,.
	\end{aligned}
\end{equation*}
These connections are $g$- resp.\ $h$-metric, hence they lift to connections of $\Sigma_gM$ resp.\ $\Sigma_hM$, where they coincide with
\begin{equation*}
	{^g}\widetilde{\nabla}^h := \beta^h_g\circ\widetilde{\nabla}^h\circ\beta^g_h :\Gamma_{C^{\infty}}(M,\Sigma_gM)\longrightarrow \Gamma_{C^{\infty}}(M,T^*\otimes \Sigma_gM)
\end{equation*}
and
\begin{equation*}
{^h}\widetilde{\nabla}^g := \beta^g_h\circ\widetilde{\nabla}^g\circ\beta^h_g :\Gamma_{C^{\infty}}(M,\Sigma_hM)\longrightarrow \Gamma_{C^{\infty}}(M,T^*\otimes \Sigma_hM)\,,
\end{equation*}respectively.

For every $X\in TM$ the difference 
\begin{equation}
	T(X):=T_{h,g}(X):={^g}\nabla^h_X-\nabla^g_X
\end{equation}
is an endomorphism of the corresponding tangent space. Since both connections are $g$-metric, $T(X)$ is skew-symmetric w.r.t.\ $g$ and we can project it to an element $\widetilde{T}(X):=\widetilde{T}_{h,g}(X)$ of the Clifford algebra $\C\ell(TM,g)$ under the map
\begin{align}\label{eqn:projection-from-end-to-clifford}
	\mathrm{pr}=	\mathrm{pr}_g:\End(TM)\cong T^*M\otimes TM \overset{\sharp^g\otimes\Id}{\longrightarrow} TM\otimes TM \overset{\pi}{\longrightarrow} \C l(M,g)\,.
\end{align}
Here, $\pi$ is the restriction of the projection from the tensor algebra bundle to the Clifford algebra bundle. 
\begin{remark}\label{rem:skewed-connection}
	\begin{enumerate}
		\item[(i)] 	Using the unique extension of each connection to the tensor algebra, it is easy to see that $T_{h,g}(X)$ is given by the following expression,
					\begin{align*}
					T_{h,g}(X) = \mathscr A^{1/2}\circ\left(\nabla^g_X \mathscr A^{-1/2} + \left(\nabla^h_X-\nabla^g_X\right)\circ \mathscr A^{-1/2}  \right)\,.
					\end{align*}
		\item[(ii)] It is proved in \cite[p.~1042]{NM07} that
				\[
					{^g}\widetilde{\nabla}^h_X - \widetilde{\nabla}^g_X = \frac 14 \widetilde{T}_{h,g}(X)\underset{g}{\cdot}\,.
				\]
	\end{enumerate}

\end{remark}

We define the transformed Dirac operators
\begin{equation*}
	D_{g,h} = I_{h,g}^{-1}D_gI_{h,g} : \Gamma_{C^{\infty}}(M,\Sigma_hM)\longrightarrow \Gamma_{C^{\infty}}(M,\Sigma_hM)
\end{equation*}
and
\begin{equation*}
	D_{h,g} = I_{g,h}^{-1}D_hI_{g,h}: \Gamma_{C^{\infty}}(M,\Sigma_gM)\longrightarrow \Gamma_{C^{\infty}}(M,\Sigma_gM)\,.
\end{equation*}
In general, $D_{g,h}$ ($D_{h,g}$) is \textit{not} $D_h$ ($D_g$). Rather, it is an operator acting canonically on $h$-spinors ($g$-spinors) but having the same spectrum as $D_h$ ($D_g$).

By \cite[Th\'{e}or\`{e}me~20]{BG92} (see also \cite[Theorem~A.8]{NM07}) these transformed Dirac operators are given by the following expressions,
\begin{equation}\label{eqn:definition-D_h,g}
	\begin{aligned}
		D_{h,g}\varphi & = \sum_{i=1}^n e_i\underset{g}{\cdot} \widetilde{\nabla}^g_{\mathscr A^{-1/2}e_i}\varphi + \frac 14 \sum_{i=1}^n e_i\underset{g}{\cdot} \widetilde{T}_{h,g}(\mathscr A^{-1/2}e_i) \underset{g}{\cdot} \varphi\,,\\
		D_{g,h}\psi & = \sum_{i=1}^n v_i\underset{h}{\cdot} \widetilde{\nabla}^g_{\mathscr A^{1/2}v_i}\psi + \frac 14 \sum_{i=1}^n v_i\underset{h}{\cdot} \widetilde{T}_{g,h}(\mathscr A^{1/2}v_i) \underset{h}{\cdot} \psi\,,\\
	\end{aligned}
\end{equation}where $(e_1,\ldots,e_n)$ is a local $g$-ONB and $(v_1,\ldots,v_n)$ a local $h$-ONB.

\begin{remark}
	The formulae \eqref{eqn:definition-D_h,g} differ from the statement of \cite[Th\'{e}or\`{e}me~20]{BG92} in that the factor $\tfrac 14$ in front of the second sum has to be replaced by a $\tfrac 12$. As explained in \cite[Remark~A.10]{NM07}, this is due to \cite{BG92} using $\sigma\otimes\tau\mapsto \tfrac 12 \sigma\cdot\tau$ as the convention for Clifford multiplication.
\end{remark}

The following two weight functions will be central to our main result:
\begin{equation}
\begin{aligned}\label{eqn:definition-delta-omega}
\delta:=\delta_{g,h} : M&\longrightarrow \R\\
x & \mapsto 2\sinh\left(\frac{n}{4}\max_{\lambda\in\spec ( \mathscr A_{g,h}(x) ) }\left|\ln\lambda\right|\right) = \max_{\lambda\in\spec ( \mathscr A_{g,h}(x) ) }\left|\lambda^{n/4} -\lambda^{-n/4}\right|\,,\\
\omega:=\omega_{g,h} : M&\longrightarrow \R\\
x&\mapsto |\nabla^h-\nabla^g|_g(x)\,,
\end{aligned}
\end{equation}
where we view $\nabla^h-\nabla^g$ as a section of the bundle $\Hom(TM,\End(TM))$ and $|\cdot|_g$ is the corresponding operator norm induced by the inner product $g$. \vspace{2mm}

\textit{In the sequel we will always assume that $\omega_{g,h}$ is bounded.}\vspace{2mm}

The following proposition establishes, in view of Remark \ref{rem:skewed-connection}(i), the connection between $T_{h,g}={}^g\nabla^h-\nabla^g$ and $\nabla^h-\nabla^g$, by showing that any pointwise bound on the latter will imply one on the former. Note that the statement is independent of the quasi-isometry of the metrics $g$ and $h$.

\begin{proposition}\label{prop:nabla-difference-bound} For all vector fields $X$ on $M$ one has the pointwise estimate
	\[
	|\nabla^g_X \mathscr A^{-1/2}|_g \leq |\mathscr A^{-1}|^{3/2} |\mathscr A| |\nabla^h_X-\nabla^g_X|_g\,,
	\] in particular,
	\[
	|T_{h,g}|_g \leq (|\mathscr A^{-1}|^{3/2}|\mathscr A|^{3/2}+|\mathscr A^{-1}|^{1/2}|\mathscr A|^{1/2})\omega_{g,h}\,.
	\]
\end{proposition}
\begin{proof}
	We covariantly differentiate the identity
	\[
	\Id = \mathscr A^{1/2}\circ \mathscr A^{-1/2}
	\]	to obtain
	\[
	0 = \nabla_X^g\mathscr A^{1/2}\circ \mathscr A^{-1/2} + \mathscr A^{1/2}\circ \nabla_X^g\mathscr A^{-1/2}\,.
	\]
	Rewriting this gives
	\begin{align}\label{eqn:nabla-difference-bound-01}
	\nabla^g_X\mathscr A^{-1/2} = -\mathscr A^{-1/2}\circ \nabla^g_X\mathscr A^{1/2}\circ \mathscr A^{-1/2}\,,
	\end{align}
	which shows that any bound on $\nabla^g_X\mathscr A^{1/2}$ leads to a bound on $\nabla^g_X\mathscr A^{-1/2}$.
	
	Next, we differentiate the identity
	\[
	g(\mathscr A^{1/2}Y,\mathscr A^{1/2}Y) = h(Y,Y)
	\] in direction $X$, which yields
	\begin{align*}	
	& 2g(\nabla^g_X(\mathscr A^{1/2}Y),\mathscr A^{1/2}Y)= 2g(\nabla^g_X\mathscr A^{1/2}(Y),\mathscr A^{1/2}Y) + 2g(\mathscr A^{1/2}\nabla^g_XY,\mathscr A^{1/2}Y)\\ =& 2 h (\nabla^h_X Y,Y)\,.
	\end{align*}
	This, in turn, implies by definition of $\mathscr A$
	\[
	g(\mathscr A^{1/2}\circ \nabla_X^g\mathscr A^{1/2}(Y),Y) = g(\mathscr A\circ(\nabla_X^h-\nabla_X^g)(Y),Y)\,.
	\]
	
	A simple calculation shows that the self-adjointness of $\mathscr A^{1/2}$ implies that of $\nabla^g_X\mathscr A^{1/2}$. Fix a point $x\in M$, let $\lambda$ be an eigenvalue of $\nabla^g_X\mathscr A^{1/2}$ at $x$ with $|\nabla^g_X\mathscr A^{1/2}|_g=|\lambda|$ and $u\in T_xM$ a corresponding $g$-normalized eigenvector. Then we have
	\begin{align*}
	|g(\mathscr A^{1/2}\circ \nabla_X^g\mathscr A^{1/2}(u),u)| = |\lambda| |g(\mathscr A^{1/2}u,u)| = |\nabla_X^g\mathscr A^{1/2}|_g |g(\mathscr A^{1/2}u,u)|\geq |\nabla_X^g\mathscr A^{1/2}|_g |\mathscr A^{-1}|^{-1/2}\,,
	\end{align*}
	so that
	\begin{align*}
	&|\nabla_X^g\mathscr A^{1/2}|_g\leq |\mathscr A^{-1}|^{1/2}\sup_{|v|\leq 1}|g(\mathscr A^{1/2}\circ \nabla_X^g\mathscr A^{1/2}(v),v)| = |\mathscr A^{-1}|^{1/2}\sup_{|v|\leq 1}|g(\mathscr A\circ(\nabla_X^h-\nabla_X^g)(v),v)|\\ \leq & |\mathscr A^{-1}|^{1/2} |\mathscr A| |\nabla^h_X-\nabla^g_X|_g\,.
	\end{align*}
	Combining this with \eqref{eqn:nabla-difference-bound-01} proves the proposition.	
\end{proof}

\section{Dirac HPW-Formulae}\label{hpwform}

The goal of this section is to prove Theorems~\ref{thm:dirac-hpw-formula} and \ref{thm:dirac-hpw-formula2} below, which are decomposition forumulae that calculate the parabolic variant of the central operator of the Belopol'skii-Birman-Theorem, which in turn is one of the key ingredients in the proof of our main results. As this approach to scattering can be traced back to the case of the scalar Laplace-Beltrami-Operator considered in \cite{hempel} (where an elliptic approach is followed), we call these formulae \emph{HPW-formulae}.

First, we dissect the various Dirac operators we have defined in the last section by writing them as compositions of covariant derivatives and certain homomorphism fields. Define

\begin{align*}
	L_g\in \Gamma_{\ICC}(M,\Hom(T^*M&\otimes\Sigma_gM,\Sigma_gM)) & & & L_h\in \Gamma_{\ICC}(M,\Hom(T^*M&\otimes\Sigma_hM,\Sigma_hM)) \\
	L_g(\xi\otimes \sigma)&= \xi^{\sharp^g}\underset{g}{\cdot}\sigma\,, & & & L_h(\xi\otimes \sigma)&= \xi^{\sharp^h}\underset{h}{\cdot}\sigma\,,\\
	L_{h,g}\in \Gamma_{\ICC}(M,\Hom(T^*M&\otimes\Sigma_gM,\Sigma_gM)) & & & L_{g,h}\in \Gamma_{\ICC}(M,\Hom(T^*M&\otimes\Sigma_hM,\Sigma_hM))\\
	L_{h,g}(\xi\otimes \sigma)&=  (\mathscr A'^{-1/2}\xi)^{\sharp^g}\underset{g}{\cdot}\sigma \,, & & & L_{g,h}(\xi\otimes \sigma)&=  (\mathscr A'^{1/2}\xi)^{\sharp^h}\underset{h}{\cdot}\sigma\,.
\end{align*}
We will also need the following multiplication operators
\begin{align*}
M_{h,g}&\in \Gamma_{\ICC}(M,\End(\Sigma_gM)) & & & M_{g,h}&\in \Gamma_{\ICC}(M,\End(\Sigma_hM))\\
M_{h,g}(\sigma) &= \frac{1}{4} \sum_{i=1}^{n}e_i \underset{g}{\cdot} \widetilde{T}_{h,g}(\mathscr A^{-1/2}e_i) \underset{g}{\cdot} \sigma  \,, & & & M_{g,h}(\sigma) & = \frac 14 \sum_{i=1}^n v_i\underset{h}{\cdot}  \widetilde{T}_{g,h}(\mathscr A^{1/2}v_i) \underset{h}{\cdot}\sigma\,,\\
\end{align*}
where $(e_1,\ldots,e_n)$ and $(v_1,\ldots,v_n)$ are a $g$-ONB and a $h$-ONB, respectively, at the appropriate base point. It is easy to see that the endomorphism $M_{h,g}$ is in general not normal. Indeed, the expression $\sum_{i=1}^{n}e_i \cdot \widetilde{T}_{h,g}(\mathscr A^{-1/2}e_i)\in\C\ell(M,g)$ is a linear combination of terms of degree one or three. We therefore dissect $M_{h,g}$ into its selfadjoint and anti-selfadjoint parts,
\begin{align*}
M_{h,g} = M_{h,g}^{+}+M_{h,g}^{-} \quad\text{ with }\quad  M_{h,g}^{+} = \frac 12(M_{h,g}+M_{h,g}^*) \quad\text{ and }\quad M_{h,g}^{-} = \frac 12(M_{h,g}-M_{h,g}^*)\,,
\end{align*}
and analogously for $M_{g,h}$.

By \eqref{eqn:definition-D_h,g}, the various Dirac operators can now be written as

\begin{align*}
	D_g & = L_g\widetilde{\nabla}^g ,& & &D_h & = L_h\widetilde{\nabla}^h\,,\\
	D_{h,g} & = {L_{h,g}}\widetilde{\nabla}^g + M_{h,g},& & & {D_{g,h}} & = {L_{g,h}}\widetilde{\nabla}^h+{M_{g,h}}\,.
\end{align*}

\begin{lemma}\label{lem:fibrewise_adjoint_L}
	The fibrewise adjoints $L_h^*$ and $L_{g,h}^*$ of $L_h$ respectively $L_{g,h}$ are given by
	\begin{align*}
		L_h^*(\tau)&= -\sum_{i=1}^n \varphi_i\otimes v_i\underset{h}{\cdot}\tau\,, & & & L_{g,h}^*(\tau)&= -\sum_{i=1}^n \mathscr A'^{1/2}\varphi_i\otimes v_i\underset{h}{\cdot}\tau\,,\\
	\end{align*}
	where $(v_1,\ldots,v_n)$ is an $h$-ONB of $TM$ at the appropriate basepoint and $(\varphi_1,\ldots,\varphi_n)$ the corresponding $h$-dual ONB. Analogous formulae hold for $L_g^*$ and $L_{h,g}^*$.
\end{lemma}
\begin{proof}
	We calculate straightforwardly
	\begin{align*}
		\left( L_h(\xi\otimes\sigma),\tau \right) =& \left( \xi^{\sharp}\cdot\sigma,\tau \right) = \sum_{i=1}^n \left( \xi(v_i)\varphi_i^{\sharp}\cdot\sigma,\tau \right) =\sum_{i=1}^n\xi(v_i)\left( v_i\cdot\sigma,\tau   \right)\\=&-\sum_{i=1}^n  \left(\xi,\varphi_i\right)	\left( \sigma,v_i\cdot \tau   \right)
		= -\sum_{i=1}^n \left( \xi\otimes\sigma,\varphi_i\otimes v_i\cdot\tau   \right) \,,
	\end{align*}
	which proves the formula for $L_h^*$. The one for $L_{g,h}^*$ is obtained by precomposing $L_h$ with $\mathscr A'^{1/2}\otimes\Id$ and using the first relation in \eqref{eqn:spectral-calculus-of-A} and that $\mathscr A'^{1/2}$ is an isometry 
$(T^*M,g)\longrightarrow (T^*M,h)$.
\end{proof}

Now define the following smooth endomorphism fields
\begin{align*}
	K_g \in \Gamma_{\ICC}&(M,\End(T^*M\otimes \Sigma_gM)) & & & K_h \in \Gamma_{\ICC}&(M,\End(T^*M\otimes \Sigma_hM))\\
	K_g(\xi\otimes\sigma)&= -\sum_{i=1}^n \varepsilon_i\otimes e_i \underset{g}{\cdot} \xi^{\sharp^g}\underset{g}{\cdot}\sigma\,, & & & K_h(\xi\otimes\sigma)&= -\sum_{i=1}^n \varphi_i\otimes v_i \underset{h}{\cdot} \xi^{\sharp^h}\underset{h}{\cdot}\sigma\,.
\end{align*}

Denote  $\widetilde{\mathscr A'}^\alpha = \mathscr A'^\alpha\otimes\Id$, which we use as a symbol for the corresponding sections of endomorphisms of $T^*M\otimes\Sigma_gM$ and $T^*M\otimes\Sigma_hM$.

\begin{lemma}\label{lem:K_j-sa-and-commutation-formula}
	The endomomorphisms $K_g$ and $K_h$ are fibrewise selfadjoint and satisfy 
	\begin{align}\label{eqn:K_gK_h-commutation-relation}
		\left(\mathscr A'^{1/2}\otimes\beta^g_h\right) K_g=K_h\left(\mathscr A'^{1/2}\otimes\beta^g_h\right)\,.
	\end{align}
	Moreover, 
	\begin{equation}\label{eqn:fibrewise-L-difference}
		\begin{aligned}
				L_h^* \beta^g_h L_{h,g} &=  K_h\widetilde{\mathscr A'}^{-1/2}  \left(\mathscr A'^{1/2}\otimes \beta^g_h\right) = \left(\mathscr A'^{1/2}\otimes \beta^g_h\right)  K_g \widetilde{\mathscr A'}^{-1/2} \,,\\
				 L_{g,h}^* \beta^g_h L_g &=  \widetilde{\mathscr A'}^{1/2} K_h \left(\mathscr A'^{1/2}\otimes \beta^g_h\right) = \left(\mathscr A'^{1/2}\otimes \beta^g_h\right)  \widetilde{\mathscr A'}^{1/2} K_g \,.
		\end{aligned}
	\end{equation}
\end{lemma}
\begin{proof}
	We have
	\begin{multline*}
		\left( K_g(\xi\otimes\sigma),\xi\otimes\sigma\right) = -\sum_{i=1}^n \left( \varepsilon_i,\xi\right)\left( e_i\cdot \xi^\sharp\cdot \sigma,\sigma\right) = -\sum_{i,j=1}^n \left( \varepsilon_i,\xi\right)\left( \varepsilon_j,\xi\right)\left( e_i\cdot e_j\cdot \sigma,\sigma\right)  \\= \sum_{i,j=1}^n \left( \varepsilon_i,\xi\right)\left( \varepsilon_j,\xi\right)\left( e_j\cdot \sigma,e_i\cdot \sigma\right)
		=|\xi|^2|\sigma|^2+2\sum_{i<j}\left( \varepsilon_i,\xi\right)\left( \varepsilon_j,\xi\right)\Re\left( e_j\cdot \sigma,e_i\cdot \sigma\right)\,,
	\end{multline*}
	which is real. Hence, $K_g$ is selfadjoint. The calculation for $K_h$ is entirely analogous.
	
	Assume w.l.o.g.\ that $\mathscr A^{-1/2}e_i=v_i$ for all $1\leq i\leq n$. Then we also have $\mathscr A'^{1/2}\varepsilon_i=\varphi_i$. Hence
	\begin{align*}
		(\mathscr A'^{1/2}\otimes\beta^g_h)K_g(\xi\otimes\sigma) &= -\sum_{i=1}^{n} \mathscr A'^{1/2}(\varepsilon_i)  \otimes\beta^g_h(e_i\underset{g}{\cdot}\xi^{\sharp^g}\underset{g}{\cdot}\sigma) = -\sum_{i=1}^{n} \varphi_i  \otimes  \mathscr A^{-1/2}e_i \underset{h}{\cdot}  \mathscr A^{-1/2 }\xi^{\sharp^g}\underset{h}{\cdot}\beta^g_h(\sigma)\\
		&=  -\sum_{i=1}^{n} \varphi_i  \otimes  v_i \underset{h}{\cdot}  \mathscr A^{-1/2 }\xi^{\sharp^g}\underset{h}{\cdot}\beta^g_h(\sigma) =  -\sum_{i=1}^{n} \varphi_i  \otimes  v_i \underset{h}{\cdot}   \mathscr A'^{1/2 }\xi^{\sharp^h}\underset{h}{\cdot}\beta^g_h(\sigma)\\
		&=K_h (\mathscr A'^{1/2}\otimes\beta^g_h)(\xi\otimes\sigma)\,,
	\end{align*}  
	where we have used the first relation in \eqref{eqn:spectral-calculus-of-A}.
	
	We prove the first equality in the first line of \eqref{eqn:fibrewise-L-difference}. By the definition of $L_h$ and $L_{h,g}$ and by Lemma~\ref{lem:fibrewise_adjoint_L} we have
	\begin{align*}
		L_h^* \beta^g_h L_{h,g}(\xi\otimes\sigma) &= L_h^*\beta^g_h( (\mathscr A'^{-1/2}\xi)^{\sharp^g}\underset{g}{\cdot}\sigma) = L_h^*( \mathscr A^{-1/2}(\mathscr A'^{-1/2}\xi)^{\sharp^g}\underset{h}{\cdot}\beta^g_h(\sigma))\\
		&= L_h^*( \xi^{\sharp^h}\underset{h}{\cdot}\beta^g_h(\sigma)) =  -\sum_{i=1}^n \varphi_i\otimes v_i\underset{h}{\cdot}\xi^{\sharp^h}\underset{h}{\cdot}\beta^g_h(\sigma)\\
		&= K_h \widetilde{\mathscr A'}^{-1/2} \left(\mathscr A'^{1/2}\otimes\beta^g_h\right)\,, 
	\end{align*}
	where we have used once more \eqref{eqn:spectral-calculus-of-A}. The second equality in the first line of \eqref{eqn:fibrewise-L-difference} follows from \eqref{eqn:K_gK_h-commutation-relation}. The proof of the second line in \eqref{eqn:fibrewise-L-difference} is analogous.
\end{proof}

\begin{remark}
	If we identify $T^*M$ with $TM$ using the metric $j\in\{g,h\}$, we recognize $K_j$ from its definition as a multiple of the projection onto the orthogonal complement of the kernel of Clifford multiplication, cf.~\cite[p.~69]{Fr00}, from which we could have also deduced selfadjointness.
\end{remark}



To state the main results of this section we need to introduce several functions, sections and operators. To this end, we denote by $\mabs:\C\to \R$ the absolute value function and by $\sgn:\C\to  \C$ the sign-function with $\sgn(0)=1$ so that we have for every diagonalizable operator $B$ on a finite dimensional vector space a decomposition $B=\mabs(B)\sgn(B)$ in which $\mabs(B)$ and $\sgn(B)$ commute, the eigenvalues of $\mabs(B)$ are nonnegative and the eigenvalues of $\sgn(B)$ have modulus one. Note that if $B$ is normal with respect to a distinguished inner product, then this is the usual polar decomposition, i.e., $\mabs(B)$ is nonnegative and $\sgn(B)$ is unitary.

\begin{align*}
	S_{g,h}:M& \longrightarrow \R &&&  \widetilde{S}_j&\in\Gamma_{C^{\infty}}(M,\End(T^*M\otimes \Sigma_jM))\\
	x&\mapsto \varrho_{g,h}(x)^{1/2}-\varrho_{g,h}(x)^{-1/2}\,, &&& \widetilde{S}_j(x) &= \widetilde{\mathscr{A}'}(x)^{-1/2} - \Id(x) = (\mathscr A'(x)^{-1/2}-\Id_{T^*M}(x))\otimes \Id_{\Sigma_jM}(x)\,,
\end{align*}
\begin{align*}
\widehat{S}_{j}&\in \Gamma_{\ICC}(M,\End(T^* M\otimes\Sigma_jM))\\
\widehat{S}_{j}(x) &= \varrho_{g,h}(x)^{1/2} K_j(x)\widetilde{\mathscr A'}(x)^{-1/2}-\varrho_{g,h}(x)^{-1/2} \widetilde{\mathscr A'}(x)^{1/2}K_j(x)\,.
\end{align*}

Note that $\widehat{S}_{j}$ is fibrewise similar to the selfadjoint endomorphism 
\[
	\varrho_{g,h}^{1/2} \widetilde{\mathscr A'}^{-1/4}K_j\widetilde{\mathscr A'}^{-1/4}-\varrho_{g,h}^{-1/2} \widetilde{\mathscr A'}^{1/4}K_j\widetilde{\mathscr A'}^{1/4}\,,
\]
and that in light of \eqref{eqn:K_gK_h-commutation-relation} we have  
\begin{equation}\label{eqn:widehat_S-A-commutator}
(\mathscr A'^{1/2}\otimes\beta^g_h)\widehat{S}_{g}=\widehat{S}_{h}(\mathscr A'^{1/2}\otimes\beta^g_h)\,.
\end{equation}

\noindent We continue with our definitions.

\begin{align*}
	S_{g,h;j}:\Gamma_{L^2}(M,\Sigma_jM)&\longrightarrow  \Gamma_{L^2}(M,\Sigma_jM) &&& \widehat{S}_{g,h;j}:\Gamma_{L^2}(T^*M\otimes\Sigma_jM)&\longrightarrow  \Gamma_{L^2}(T^* M\otimes\Sigma_jM)\\
	S_{g,h;j}\varphi(x)&= \mabs(S_{g,h}(x))^{1/2}\varphi(x)\,, &&& \widehat{S}_{g,h;j}\varphi(x)&= \mabs(\widehat{S}_j(x))^{1/2}\varphi(x)\,,\\
	Q_g : \Gamma_{L^2}(M,T^*M\otimes &\Sigma_gM)\longrightarrow \Gamma_{L^2}(M,T^*M\otimes \Sigma_gM) &&& Q_h : \Gamma_{L^2}(M,\Sigma_hM)\longrightarrow &\,\Gamma_{L^2}(M,T^*M\otimes \Sigma_hM)\\
	Q_g\varphi(x) & = \mabs(\widetilde{S}_g(x))^{1/2}\varphi(x)\,, &&& Q_h\varphi(x)  =  &\mabs(\widetilde{S}_h(x))^{1/2}L_h(x)^*\varphi(x)\,,\\
	R_g : \Gamma_{L^2}(M,\Sigma_gM)&\longrightarrow \Gamma_{L^2}(M,T^*M\otimes \Sigma_gM) &&& R_h : \Gamma_{L^2}(M,\Sigma_hM)\longrightarrow &\,\Gamma_{L^2}(M,T^*M\otimes \Sigma_hM)\\
	R_g\varphi(x) & = \frac{1}{4}\widetilde{T}_{h,g}(x)\varphi(x)\,, &&& R_h\varphi(x)  =& \mabs(\widetilde{S}_h(x))L_h(x)^*\varphi(x)\,,
\end{align*}
\begin{align*}
U_{g,h}:\Gamma_{L^2}(M,\Sigma_gM)&\longrightarrow  \Gamma_{L^2}(M,\Sigma_hM)\\
U_{g,h}\varphi(x)&= \sgn(S_{g,h}(x))\varrho_{g,h}(x)^{-1/2}\beta^g_h(\varphi(x))\,,
\end{align*}
\begin{align*}
\widehat{U}_{g,h}:\Gamma_{L^2}(M,T^*M\otimes \Sigma_gM)&\longrightarrow  \Gamma_{L^2}(M,T^*M\otimes \Sigma_hM)\\
\widehat{U}_{g,h}\varphi(x)&= \varrho_{g,h}(x)^{-1/2}\sgn(\widehat{S}_{h}(x))\widetilde{I}_{g,h}(\varphi)(x) \\
&=\varrho_{g,h}(x)^{-1/2}\widetilde{I}_{g,h}(\sgn(\widehat{S}_{g})(\varphi))(x)\,.
\end{align*}
By $g\sim h$, the operators $S_{g,h;j}$, $\widehat{S}_{g,h;j}$, $Q_j$, $R_h$, $U_{g,h}$ and $\widehat{U}_{g,h}$ are bounded. For $R_g$ to be bounded, the boundedness of $\omega_{g,h}$ is additionally needed, see Lemma~\ref{lem:tilde-T-bounds} below. Moreover, $U_{g,h}$ is always unitary whereas $\widehat{U}_{g,h}$ is only unitary if $\widehat{S}_j$ is selfadjoint for one (and then both) $j\in\{g,h\}$.

Next, define for $i\in\{+,-\}$ the operators
\begin{align*}
	T^i_g : \Gamma_{L^2}(M,\Sigma_gM)&\longrightarrow  \Gamma_{L^2}(M,\Sigma_gM) &&& T^i_{g;h} : \Gamma_{L^2}(M,\Sigma_hM)&\longrightarrow  \Gamma_{L^2}(M,\Sigma_hM)\\
	T^i_g\varphi(x) & = \mabs(M^i_{h,g}(x))^{1/2}\varphi(x)\,, &&& T^i_{g;h}\varphi(x) & = \beta^g_h(\mabs(M^i_{h,g}(x))^{1/2}\beta^h_g(\varphi(x)))\,,\\
	T^i_h : \Gamma_{L^2}(M,\Sigma_hM)&\longrightarrow  \Gamma_{L^2}(M,\Sigma_hM) &&& T^i_{h;g} : \Gamma_{L^2}(M,\Sigma_gM)&\longrightarrow  \Gamma_{L^2}(M,\Sigma_gM)\\
	T^i_h\varphi(x) & = \mabs(M^i_{g,h}(x))^{1/2}\varphi(x)\,, &&& T^i_{h;g}\varphi(x) & = \beta^h_g(\mabs(M^i_{g,h}(x))^{1/2}\beta^g_h(\varphi(x)))\,,
\end{align*}
\begin{align*}
	V^i_{g,h} : \Gamma_{L^2}(M,\Sigma_gM) & \longrightarrow  \Gamma_{L^2}(M,\Sigma_hM) &&& \widehat{V}^i_{g,h} : \Gamma_{L^2}(M,\Sigma_gM) & \longrightarrow  \Gamma_{L^2}(M,\Sigma_hM)\\
	V^i_{g,h}\varphi(x) & = \beta^g_h(\sgn(M^i_{h,g}(x))\varphi(x))\,, &&& \widehat{V}^i_{g,h}\varphi(x) = \varrho_{g,h}^{-1}&(x)\overline{\sgn}(M^i_{g,h}(x))\beta^g_h(\varphi(x))\,,\\
	W^i_{g,h} : \Gamma_{L^2}(M,\Sigma_gM) & \longrightarrow \Gamma_{L^2}(M,\Sigma_hM) &&& \widehat{W}_{g,h}:\Gamma_{L^2}(M,T^*M\otimes \Sigma_gM) & \longrightarrow \Gamma_{L^2}(M,T^*M\otimes \Sigma_hM)\\
	W^i_{g,h}\varphi(x) &= \sgn(M^i_{g,h}(x))\beta^g_h(\varphi(x))\,, &&& \widehat{W}_{g,h}\varphi(x)& = \sgn(\widetilde{S}_h(x))\widetilde{I}_{g,h}(\varphi)(x)\,,
\end{align*}
where $\overline{\sgn}$ is the complex conjugate of $\sgn$.

The operators $T^i_g, T^i_{g;h},T^i_h$ and $T^i_{h;g}$ are bounded in view of $g\sim h$ and the boundedness of $\omega_{g,h}$, see Corollary~\ref{cor:M-bounds} below. For the operators $V^i_{g,h}, \widehat{V}^i_{g,h}$, $W^i_{g,h}$ and $\widehat{W}_{g,h}$ to be bounded, $g\sim h$ is sufficient. \vspace{2mm}

We recall that if $g$ is a geodesically complete metric on $M$, then $D_g$ as well as all its powers are essentially self-adjoint in $\Gamma_{L^2}(M,\Sigma_gM)$, when defined initially on smooth compactly supported spinors and the corresponding unique self-adoint realizations will be denoted with the same symbol again. We denote by 
$$
(P^g_s)_{s>0}:=(\exp(-s D_g^2))_{s>0}
$$ 
the heat semigroup associated with $D_g^2$ in $\Gamma_{L^2}(M,\Sigma_gM)$, defined via the spectral calculus of $D_g^2$. Note that $P^g_s$ is precisely the operator $f(D)$ defined by the spectral calculus of $D$, where $f:\IR\to\IR$ is given by $f(\lambda):=e^{-t\lambda^2}$.

With these definitions, the central results of this section are given by the following two results below:

\begin{theorem}[Dirac-HPW-formula I]\label{thm:dirac-hpw-formula}
	Let $g\sim h$ be geodesically complete Riemannian metrics on $M$ such that the function $\omega_{g,h}$ is bounded. Given $s>0$ define the bounded operator 
	$$
	\mathscr{T}_{g,h,s}:\Gamma_{L^2}(M,\Sigma_gM)\longrightarrow  \Gamma_{L^2}(M,\Sigma_hM)
	$$
	by
	\begin{multline*}
		\mathscr{T}_{g,h,s}:=  (\widehat{S}_{g,h;h}\widetilde{\nabla}^hP^h_s)^*\widehat{U}_{g,h}\widehat{S}_{g,h;g}\widetilde{\nabla}^gP^g_s +(T^+_{g;h}D_hP^h_s)^*V^+_{g,h} T^+_gP^g_s+(T^-_{g;h}D_hP^h_s)^*V^-_{g,h} T^-_gP^g_s
		\\ -(T_h^+ P^h_s)^*\widehat{V}^+_{g,h}T^+_{h;g}D_gP^g_s-(T_h^- P^h_s)^*\widehat{V}^-_{g,h}T^-_{h;g}D_gP^g_s - (S_{g,h;h}P^h_s)^*U_{g,h}S_{g,h;g}P^g_{s/2}D_g^2 P^g_{s/2}\,.
	\end{multline*}
	Then the following formula holds for all $s>0$, $\varphi\in\Dom(D^2_g)$ and $\psi\in\Dom(D_h^2)$,
	\begin{align}\label{lem1:eqn:BB}
		\langle \psi,\mathscr{T}_{g,h,s}\varphi\rangle
		= \langle D_h^2\psi,P^h_sI_{g,h}P^g_s\varphi\rangle - \langle \psi,P^h_sI_{g,h}P^g_sD_g^2\varphi\rangle\,.
	\end{align}
\end{theorem}
\begin{proof} Since $D^2_g$ and $D^2_h$ are essentially self-adjoint (so that smooth compactly supported spinors are dense with respect to the corresponding graph norms), we can assume that $\varphi$ and $\psi$ are smooth and compactly supported. We add 
$$
0 = \langle I_{g,h}^{-1}P^h_s\psi,D_g^2P^g_s\varphi \rangle-\langle I_{g,h}^{-1}P^h_s\psi,D_g^2P^g_s\varphi \rangle
$$
 to the right hand side of \eqref{lem1:eqn:BB} and obtain
	\begin{align}
		&\langle D_h^2\psi,P^h_sI_{g,h}P^g_s\varphi\rangle - \langle \psi,P^h_sI_{g,h}P^g_sD_g^2\varphi\rangle\nonumber\\
		=&\langle D_h^2\psi,P^h_sI_{g,h}P^g_s\varphi\rangle - \langle I_{g,h}^{-1} P^h_s\psi,D_g^2P^g_s\varphi\rangle - \langle P^h_s\psi,I_{g,h}P^g_sD_g^2\varphi\rangle + \langle I_{g,h}^{-1} P^h_s\psi,D_g^2P^g_s\varphi\rangle\nonumber\\
		=&\langle D_hP^h_s\psi,D_hI_{g,h}P^g_s\varphi\rangle - \langle D_gI_{g,h}^{-1} P^h_s\psi,D_gP^g_s\varphi\rangle - \langle P^h_s\psi,(I_{g,h}-(I_{g,h}^{-1})^*)D_g^2P^g_s\varphi\rangle\label{eqn:dirac-hpw-formula-01}
	\end{align}	
	
	We transform the last term in \eqref{eqn:dirac-hpw-formula-01} as follows
	\begin{align*}
		&\langle P^h_s\psi,(I_{g,h}-(I_{g,h}^{-1})^*)D_g^2P^g_s\varphi\rangle = \int_{M}\!\left( P^h_s\psi,(\beta^g_h-\varrho_{g,h}^{-1}\beta^g_h)D_g^2 P^g_s\varphi \right) \dd\mu_h\\
		&=\int_{M}\!\left( P^h_s\psi,(1-\varrho_{g,h}^{-1})\beta^g_hD_g^2 P^g_s\varphi \right) \dd\mu_h=\int_{M}\!\left( P^h_s\psi,S_{g,h}\varrho_{g,h}^{-1/2}\beta^g_hD_g^2 P^g_s\varphi \right) \dd\mu_h\\
		&=\int_{M}\!\left( P^h_s\psi,\mabs(S_{g,h})^{1/2}\sgn(S_{g,h})\varrho_{g,h}^{-1/2}\mabs(S_{g,h})^{1/2} \beta^g_hD_g^2 P^g_s\varphi \right) \dd\mu_h\\
		&=\int_{M}\!\left( P^h_s\psi,\mabs(S_{g,h})^{1/2}\sgn(S_{g,h})\varrho_{g,h}^{-1/2} \beta^g_h\mabs(S_{g,h})^{1/2}D_g^2 P^g_s\varphi \right) \dd\mu_h\\
		&=\langle \psi,P^h_sS_{g,h;h}U_{g,h}S_{g,h;g}P^g_{s/2}D_g^2 P^g_{s/2}\varphi \rangle\,.
	\end{align*}
	
	Let us come back to the first two terms in \eqref{eqn:dirac-hpw-formula-01},
	\begin{align}\label{eqn:dirac-hpw-formula-02}
		&\langle D_hP^h_s\psi,D_hI_{g,h}P^g_s\varphi\rangle - \langle D_gI_{g,h}^{-1} P^h_s\psi,D_gP^g_s\varphi\rangle\nonumber\\
		=&\langle D_hP^h_s\psi,D_hI_{g,h}P^g_s\varphi\rangle - \langle I_{g,h} D_gI_{g,h}^{-1} P^h_s\psi,I_{h,g}^*D_gP^g_s\varphi\rangle\nonumber\\
		=&\langle D_hP^h_s\psi,D_hI_{g,h}P^g_s\varphi\rangle - \langle D_{g,h} P^h_s\psi,I_{h,g}^*D_gP^g_s\varphi\rangle\nonumber\\
		=&\langle L_h\tilde{\nabla}^hP^h_s\psi,D_hI_{g,h}P^g_s\varphi\rangle - \langle (L_{g,h}\widetilde{\nabla}^h+M_{g,h}) P^h_s\psi,I_{h,g}^*D_gP^g_s\varphi\rangle\nonumber\\
		=&\langle \widetilde{\nabla}^hP^h_s\psi,(L_h^*D_hI_{g,h}-L_{g,h}^*I_{h,g}^*D_g)P^g_s\varphi\rangle - \langle M_{g,h} P^h_s\psi,I_{h,g}^*D_gP^g_s\varphi\rangle\nonumber\\
		=&\langle \widetilde{\nabla}^hP^h_s\psi,(L_h^*I_{g,h}D_{h,g}-L_{g,h}^*I_{h,g}^*D_g)P^g_s\varphi\rangle - \langle \psi,(M_{g,h} P^h_s)^*I_{h,g}^*D_gP^g_s\varphi\rangle\nonumber\\
		=&\langle \widetilde{\nabla}^hP^h_s\psi,(L_h^*I_{g,h}\left(L_{h,g}\widetilde{\nabla}^g+M_{h,g}\right)-L_{g,h}^*I_{h,g}^*(L_g\widetilde{\nabla}^g))P^g_s\varphi\rangle - \langle \psi,(M_{g,h} P^h_s)^*I_{h,g}^*D_gP^g_s\varphi\rangle\nonumber\\
		=&\langle \widetilde{\nabla}^hP^h_s\psi,(L_h^*I_{g,h}L_{h,g}-L_{g,h}^*I_{h,g}^*L_g)\widetilde{\nabla}^gP^g_s\varphi\rangle +\langle \widetilde{\nabla}^hP^h_s\psi,L_h^*I_{g,h} M_{h,g}P^g_s\varphi \rangle  \nonumber\\
		-& \langle \psi,(M_{g,h} P^h_s)^*I_{h,g}^*D_gP^g_s\varphi\rangle\nonumber\\
		=&\langle \widetilde{\nabla}^hP^h_s\psi,(L_h^*I_{g,h}L_{h,g}-L_{g,h}^*I_{h,g}^*L_g)\widetilde{\nabla}^gP^g_s\varphi\rangle +\langle \psi,(D_hP^h_s)^*I_{g,h} M_{h,g}P^g_s\varphi \rangle  \\
		-& \langle \psi,(M_{g,h} P^h_s)^*I_{h,g}^*D_gP^g_s\varphi\rangle\,.\nonumber
	\end{align}

	We rewrite the first term in \eqref{eqn:dirac-hpw-formula-02} using Lemma~\ref{lem:K_j-sa-and-commutation-formula} and the relation \eqref{eqn:widehat_S-A-commutator},
	\begin{align*}
		&\langle \widetilde{\nabla}^hP^h_s\psi,(L_h^*I_{g,h}L_{h,g}-L_{g,h}^*I_{h,g}^*L_g)\widetilde{\nabla}^gP^g_s\varphi\rangle=\int_M\! \left(\widetilde{\nabla}^hP^h_s\psi, (L_h^*\beta^g_hL_{h,g}-\varrho_{g,h}^{-1}L_{g,h}^*\beta^g_hL_g)\widetilde{\nabla}^gP^g_s\varphi \right)\dd\mu_h\\
		=& \int_M\! \left(\widetilde{\nabla}^hP^h_s\psi, \varrho_{g,h}^{-1/2}(\mathscr A'^{1/2}\otimes\beta^g_h)(\varrho_{g,h}^{1/2}K_g\widetilde{\mathscr A'}^{-1/2}-\varrho_{g,h}^{-1/2}\widetilde{\mathscr A'}^{1/2}K_g)\widetilde{\nabla}^gP^g_s\varphi \right)\dd\mu_h\\
		=& \int_M\! \left(\widetilde{\nabla}^hP^h_s\psi, \varrho_{g,h}^{-1/2}(\mathscr A'^{1/2}\otimes\beta^g_h)\widehat{S}_g\widetilde{\nabla}^gP^g_s\varphi \right)\dd\mu_h\\
		=& \int_M\! \left(\widetilde{\nabla}^hP^h_s\psi, \varrho_{g,h}^{-1/2}(\mathscr A'^{1/2}\otimes\beta^g_h)\mabs(\widehat{S}_g)^{1/2}\sgn(\widehat{S}_g)\mabs(\widehat{S}_g)^{1/2}\widetilde{\nabla}^gP^g_s\varphi \right)\dd\mu_h\\
		=& \int_M\! \left(\widetilde{\nabla}^hP^h_s\psi, \mabs(\widehat{S}_h)^{1/2}\varrho_{g,h}^{-1/2}(\mathscr A'^{1/2}\otimes\beta^g_h)\sgn(\widehat{S}_g)\mabs(\widehat{S}_g)^{1/2}\widetilde{\nabla}^gP^g_s\varphi \right)\dd\mu_h\\
		=& \langle \tilde{\nabla}^hP^h_s\psi,\widehat{S}_{g,h;h}^*\widehat{U}_{g,h}\widehat{S}_{g,h;g}\tilde{\nabla}^gP^g_s\varphi\rangle\,.
	\end{align*}
	
	At last, we rewrite the second term in \eqref{eqn:dirac-hpw-formula-02} (the third is handled analogously),
	\begin{align*}
		&\langle \psi,(D_hP^h_s)^*I_{g,h} M_{h,g}P^g_s\varphi \rangle = \int_M\!\left(\psi,(D_hP^h_s)^*\beta^g_h (M^+_{h,g}+M^-_{h,g})P^g_s\varphi \right)\dd \mu_h\\
		&= \int_M\!\left(\psi,(D_hP^h_s)^*\beta^g_h \mabs(M^+_{h,g})^{1/2}\beta^h_g\beta^g_h\sgn(M^+_{h,g})\mabs(M^+_{h,g})^{1/2}P^g_s\varphi \right)\dd \mu_h\\
		&+ \int_M\!\left(\psi,(D_hP^h_s)^*\beta^g_h \mabs(M^-_{h,g})^{1/2}\beta^h_g\beta^g_h\sgn(M^-_{h,g})\mabs(M^-_{h,g})^{1/2}P^g_s\varphi \right)\dd \mu_h\\
		&=\langle\psi, (D_hP^h_s)^*(T^+_{g;h})^*V^+_{g,h}T^+_gP^g_s\varphi \rangle + \langle\psi, (D_hP^h_s)^*(T^-_{g;h})^*V^-_{g,h}T^-_gP^g_s\varphi \rangle\,.
	\end{align*}
\end{proof}

\begin{theorem}[Dirac-HPW-formula II]\label{thm:dirac-hpw-formula2}
	Let $g\sim h$ be geodesically complete Riemannian metrics on $M$ such that the function $\omega_{g,h}$ is bounded. Given $s>0$ define the bounded operator 
	$$
	\mathscr{R}_{g,h,s}:\Gamma_{L^2}(M,\Sigma_gM)\longrightarrow  \Gamma_{L^2}(M,\Sigma_hM)
	$$
	by
	\begin{multline*}
		\mathscr{R}_{g,h,s} := \left(Q_hP^h_s\right)^*\widehat{W}_{g,h} Q_g\widetilde{\nabla}^gP^g_s + \left(R_hP^h_s\right)^*\widehat{W}_{g,h}R_gP^g_s - \left(T_h^+P^h_s\right)^*W^+_{g,h}T^+_{h;g}P^g_s - \left(T_h^-P^h_s\right)^*W^-_{g,h}T^-_{h;g}P^g_s
	\end{multline*}
	Then the following formula holds for all $s>0$, $\varphi\in\Dom(D_g)$ and $\psi\in\Dom(D_h)$,
	\begin{align}\label{lem1:eqn:BB2}
		\langle \psi,\mathscr{R}_{g,h,s}\varphi\rangle= \langle D_h\psi,P^h_sI_{g,h}P^g_s\varphi\rangle - \langle \psi,P^h_sI_{g,h}P^g_sD_g\varphi\rangle\,.
	\end{align}
\end{theorem}
\begin{proof}
	As in the proof of Theorem~\ref{thm:dirac-hpw-formula}, we assume that $\varphi$ and $\psi$ are smooth and compactly supported.
	
	We start with the right hand side of \eqref{lem1:eqn:BB2},
	\begin{align}
		&\langle D_h\psi,P^h_sI_{g,h}P^g_s\varphi\rangle - \langle \psi,P^h_sI_{g,h}P^g_sD_g\varphi\rangle = \langle \psi,(D_hP^h_s I P^g_s - P^h_s I P^g_sD_g)\varphi\rangle \nonumber\\
		=&\langle\psi,(P^h_s(D_hI-ID_g)P^g_s)\varphi\rangle = \langle\psi,P^h_s(D_h-D_{g,h})IP^g_s\varphi\rangle = \langle\psi,P^h_s((L_h-L_{g,h})\widetilde{\nabla}^h-M_{g,h})IP^g_s\varphi\rangle \nonumber\\
		=& \langle\psi,P^h_s(L_h-L_{g,h})\widetilde{\nabla}^hIP^g_s\varphi\rangle - \langle\psi, P^h_sM^+_{g,h}IP^g_s\varphi\rangle- \langle\psi, P^h_sM^-_{g,h}IP^g_s\varphi\rangle\,.\label{lem1:eqn-02}
	\end{align}
	
	By Remark~\ref{rem:skewed-connection}(ii), the first term in \eqref{lem1:eqn-02} is equal to
	\begin{align*}
		&\int_M\!(\psi, P^h_s(L_h-L_{g,h})\widetilde{\nabla}^h\beta^g_hP^g_s\varphi)\dd\mu_h = \int_M\!(\psi, P^h_s(L_h-L_{g,h})(\Id\otimes\beta^g_h)\,{}^g\widetilde{\nabla}^hP^g_s\varphi)\dd\mu_h \\
		=& \int_M\!(\psi, P^h_s(L_h-L_{g,h})(\Id\otimes\beta^g_h)\widetilde{\nabla}^gP^g_s\varphi)\dd\mu_h + \int_M\!(\psi, P^h_s(L_h-L_{g,h})(\Id\otimes\beta^g_h)\tfrac 14\widetilde{T}_{h,g}P^g_s\varphi)\dd\mu_h\\
		=& \int_M\!(\psi, P^h_sL_h(\Id-\widetilde{\mathscr A'}^{1/2})(\Id\otimes\beta^g_h)\widetilde{\nabla}^gP^g_s\varphi)\dd\mu_h + \int_M\!(\psi, P^h_sL_h(\Id-\widetilde{\mathscr A'}^{1/2})(\Id\otimes\beta^g_h)\tfrac 14\widetilde{T}_{h,g}P^g_s\varphi)\dd\mu_h\\
		=& \int_M\!(\psi, P^h_sL_h\widetilde{S}_h(\mathscr A'^{1/2}\otimes\beta^g_h)\widetilde{\nabla}^gP^g_s\varphi)\dd\mu_h + \int_M\!(\psi, P^h_sL_h\widetilde{S}_h(\mathscr A'^{1/2}\otimes\beta^g_h)\tfrac 14\widetilde{T}_{h,g}P^g_s\varphi)\dd\mu_h\\
		=& \int_M\!(\psi, P^h_sL_h\mabs(\widetilde{S}_h)^{1/2}\sgn(\widetilde{S}_h)(\mathscr A'^{1/2}\otimes\beta^g_h)\mabs(\widetilde{S}_g)^{1/2}\widetilde{\nabla}^gP^g_s\varphi)\dd\mu_h \\
		+& \int_M\!(\psi, P^h_sL_h\mabs(\widetilde{S}_h)\sgn(\widetilde{S}_h)(\mathscr A'^{1/2}\otimes\beta^g_h)\tfrac 14\widetilde{T}_{h,g}P^g_s\varphi)\dd\mu_h\\
		=&\langle\psi,\left(Q_hP^h_s\right)^*\widehat{W}_{g,h} Q_g\widetilde{\nabla}^gP^g_s\varphi \rangle + \langle\psi, \left(R_hP^h_s\right)^*\widehat{W}_{g,h}R_gP^g_s\varphi\rangle\,.
	\end{align*}
	
	With $i\in\{+,-\}$, the second and third term in \eqref{lem1:eqn-02} are handled as follows,
	\begin{multline*}
		\langle\psi, P^h_sM^i_{g,h}IP^g_s\varphi\rangle = \int_M\!(\psi, P^h_sM^i_{g,h}\beta^g_hP^g_s\varphi)\dd\mu_h \\
		= \int_M\!(\psi, P^h_s\mabs(M^i_{g,h})^{1/2}\sgn(M^i_{g,h})\beta^g_h\beta^h_g\mabs(M^i_{g,h})^{1/2}\beta^g_hP^g_s\varphi)\dd\mu_h = \langle \psi,\left(T_h^iP^h_s\right)^*W^i_{g,h}T^i_{h;g}P^g_s\varphi \rangle\,.
	\end{multline*}

\end{proof}

We close this section with operator estimates that we will need in the next section.

\begin{lemma}\label{lem:S-Shat-pointwise-estimates}
	For $j\in\{g,h\}$ we have the pointwise estimates
	\begin{align*}
	|S|\leq \delta_{g,h}\,, \qquad |\widetilde{S}_j|\leq \tilde{C}_1\cdot \delta_{g,h}\,,   \quad \text{ and }\quad  |\hat{S}_j|\leq \tilde{C}_2\cdot \delta_{g,h}\,,
	\end{align*}
	where the constants $\tilde{C}_1,\tilde{C}_2$ only depend on the dimension $n$ and the constant in \eqref{eqn:quasi-isometric}.
\end{lemma}
\begin{proof}
	We prove the estimate for $\hat{S}_j$. The ones for $S$ and $\widetilde{S}_j$ will then be apparent.
	
	Fix a point $x\in M$ and let $\xi\otimes\sigma\in (T^*M\otimes \Sigma_jM)_x$, let $(e_1,\ldots,e_n)$ be a $j$-ONB of $T_xM$ consisting of eigenvectors of $\mathscr A$ with $\mathscr Ae_i=\lambda_ie_i$, $i=1,\ldots,n$, and let $(\varepsilon_1,\ldots,\varepsilon_n)$ the $j$-dual ONB which then satisfies $\mathscr A'\varepsilon_i = \lambda_i\varepsilon_i$ for all $i=1,\ldots,n$.
	
	By definition of $\hat{S}_j$ we have
	\begin{align*}
	\widetilde{\mathscr A'}^{-1/4}&\hat{S}_j\widetilde{\mathscr A'}^{1/4}(\xi\otimes\sigma)  = \left(\varrho^{1/2}\widetilde{\mathscr A'}^{-1/4}K_j\widetilde{\mathscr A'}^{-1/4}-\varrho^{-1/2}\widetilde{\mathscr A'}^{1/4}K_j\widetilde{\mathscr A'}^{1/4}\right)(\xi\otimes\sigma)\nonumber\\
	=&\varrho^{1/2}\widetilde{\mathscr A'}^{-1/4}K_j(\mathscr A'^{-1/4}\xi\otimes\sigma)-\varrho^{-1/2}\widetilde{\mathscr A'}^{1/4}K_j(\mathscr A'^{1/4}\xi\otimes\sigma)\nonumber\\
	=&\varrho^{1/2}\widetilde{\mathscr A'}^{-1/4}\left(-\sum_{i=1}^{n}\varepsilon_i\otimes e_i\cdot(\mathscr A'^{-1/4}\xi)^{\sharp}\cdot\sigma \right)-\varrho^{-1/2}\widetilde{\mathscr A'}^{1/4}\left(-\sum_{i=1}^{n}\varepsilon_i\otimes e_i\cdot(\mathscr A'^{1/4}\xi)^{\sharp}\cdot\sigma \right)\nonumber\\
	=&\sum_{i=1}^{n}\left(\varrho^{-1/2}\mathscr A'^{1/4}\varepsilon_i\otimes e_i\cdot\mathscr A^{1/4}\xi^{\sharp}\cdot\sigma-\varrho^{1/2}\mathscr A'^{-1/4}\varepsilon_i\otimes e_i\cdot\mathscr A^{-1/4}\xi^{\sharp}\cdot\sigma\right)\nonumber\\
	=&\sum_{i=1}^{n}\varepsilon_i\otimes e_i\cdot \left(\varrho^{-1/2}\lambda_i^{1/4}\mathscr A^{1/4}-\varrho^{1/2}\lambda_i^{-1/4}\mathscr A^{-1/4} \right)\left(\xi^{\sharp}\right)\cdot\sigma\,.\nonumber
	\end{align*}
	We denote the operator in parantheses
	\begin{align*}
	\mathscr B_i=\varrho^{-1/2}\lambda_i^{1/4}\mathscr A^{1/4}-\varrho^{1/2}\lambda_i^{-1/4}\mathscr A^{-1/4} = (\varrho^{-2}\lambda_i\mathscr A)^{1/4}-(\varrho^{-2}\lambda_i\mathscr A)^{-1/4}=2\sinh\left(\tfrac 14 \ln (\varrho^{-2}\lambda_i\mathscr A)\right)
	\end{align*}
	and obtain from the above, the Clifford-relations and the fact that Clifford multiplication is skew-symmetric that
	\begin{equation}\label{eqn:lem:S-Shat-pointwise-estimates-1}
	\begin{aligned}
	|\hat{S}_j(\xi\otimes\sigma)|&=\sum_{i=1}^n|{\mathscr A'}^{1/4}\varepsilon_i|| e_i\cdot\mathscr B_i((\mathscr{A}'^{-1/4}\xi)^{\sharp})\cdot\sigma|\leq |\mathscr A|^{1/4}\sum_{i=1}^n |e_i||\mathscr B_i(\mathscr{A}^{-1/4}\xi^{\sharp})||\sigma|\\&\leq |\mathscr A|^{1/4}|\mathscr A^{-1}|^{1/4}\sum_{i=1}^n \left|\mathscr B_i\right| |\xi| |\sigma|\,.
	\end{aligned}
	\end{equation}	
	It remains to bound the norm of $\mathscr B_i$. Since $\sinh$ is odd and positive for positive arguments, we have
	\begin{align*}
	\left|\mathscr B_i\right|=\left|2\sinh\left(\tfrac 14 \ln \varrho^{-2}\lambda_i\mathscr A \right)  \right|= 2\sinh\left(\tfrac 14 \left|\ln \varrho^{-2}\lambda_i\mathscr A \right|\right)\,.	
	\end{align*}
	In light of \eqref{eqn:varrho-A} we can bound the the argument of $\sinh$ as follows,
	\begin{align*}
	\tfrac 14\left|\ln\varrho^{-2}\lambda_i\mathscr A \right|=\tfrac 14 \max_{k=1,\ldots,n} \left|-\sum_{j=1}^{n}\ln\lambda_j + \ln\lambda_i + \ln\lambda_k\right|\leq \tfrac n4 \cdot \max_{k=1,\ldots,n} |\ln\lambda_k|\,.
	\end{align*}
	From this we obtain
	\begin{align*}
	|\mathscr B_i|\leq \delta_{g,h}\,,
	\end{align*}
	which in turn yields, together with \eqref{eqn:lem:S-Shat-pointwise-estimates-1}, the claimed bound on $|\hat{S}_j|$.
\end{proof}

\begin{lemma}\label{lem:tilde-T-bounds} Viewing $\widetilde{T}_{h,g}$ resp. $\widetilde{T}_{g,h}$ as sections of the bundles $\Hom(\Sigma_gM, T^*M\otimes\Sigma_gM)$ resp.~$\Hom(\Sigma_hM, T^*M\otimes\Sigma_hM))$ through Clifford multiplication, we have the pointwise estimates
	\begin{align*}
	|\widetilde{T}_{h,g}|\leq \tilde{C}_3\cdot \omega_{g,h}\quad \text{ and } \quad |\widetilde{T}_{g,h}|\leq \tilde{C}_4\cdot  \omega_{g,h}\,,
	\end{align*}
	where the constants $\tilde{C}_3$ and $\tilde{C}_4$ only depend on the dimension $n$ and the constant in \eqref{eqn:quasi-isometric}.
\end{lemma}
\begin{proof}
	Let $x\in M$ and $(e_1,\ldots,e_n)$ be a $g$-ONB of $T_xM$ consisting of eigenvectors of $\mathscr A$ with $g$-dual basis $(\varepsilon_1,\ldots,\varepsilon_n)$. If we write $T_{h,g}(e_i)=\sum_{k=1}^{n}a_{ijk}\varepsilon_j\otimes e_k$, it follows from the definition of $\widetilde{T}_{h,g}$ that
	\[
		\widetilde{T}_{h,g}(e_i) = \sum_{j,k=1}^n a_{ijk} e_j\cdot e_k\,,
	\]
	from which we obtain the estimate
	\begin{align*}
		|\widetilde{T}_{h,g}(\sigma)| = \left|\sum_{i,j,k=1}^{n}a_{ijk}\varepsilon_i\otimes e_j\cdot e_k\cdot \sigma\right| \leq \sum_{i,j,k=1}^{n}|a_{ijk}| |\varepsilon_i| |e_j\cdot e_k\cdot \sigma| = \sum_{i,j,k=1}^{n}|a_{ijk}| |\sigma|\leq C|T_{h,g}||\sigma|\,,
	\end{align*}
	where the last inequality follows from the fact that all norms on finite-dimensional vector spaces are equivalent. Together with Proposition \ref{prop:nabla-difference-bound}, this proves the first inequality in the statement of the lemma. To prove the second one, we note that analogously to the above calculation we obtain $|\widetilde{T}_{g,h}(\tau)|\leq C |T_{g,h}||\tau|$ and that
	\[
		T_{g,h}(v) = {^h}\nabla^g_{v} - \nabla^h_{v} = \mathscr A^{-1/2}\circ \nabla^g_{v}\circ \mathscr A^{1/2} - \nabla^h_{v} = \mathscr A^{-1/2}\circ(\nabla^g_{v}-{}^g\nabla^h_{v})\circ \mathscr A^{1/2} = -\mathscr A^{-1/2}\circ T_{h,g}(v)\circ \mathscr A^{1/2}\,,
	\]
	which implies
	\[
		|T_{g,h}(v)|_h\leq |\mathscr A^{-1}| |T_{h,g}(v)|_g |\mathscr A|\,.
	\]
	Hence, $|T_{g,h}|_h\leq |\mathscr A^{-1}|^{3/2}|\mathscr A||T_{h,g}|_g$ and the proof is finished.
\end{proof}

\begin{corollary}\label{cor:M-bounds}
	We have the pointwise estimates 
	\begin{align*}
		|M_{h,g}|\leq \tilde{C}_5\cdot \omega_{g,h}\quad \text{ and } \quad |M_{g,h}|\leq \tilde{C}_6\cdot  \omega_{g,h}\,,
	\end{align*}
in particular, for $i\in\{+,-\}$, 
\begin{align*}
&\left|\beta^g_h(\mabs(M^i_{h,g})^{1/2}\beta^h_g\right|=\left|\mabs(M^i_{h,g})^{1/2}\right|\leq \tilde{C}_7 \sqrt{\omega_{g,h}},\\
&\left|\beta^h_g(\mabs(M^i_{g,h})^{1/2}\beta^g_h\right|=\left|\mabs(M^i_{g,h})^{1/2}\right|\leq \tilde{C}_8 \sqrt{\omega_{g,h}},
\end{align*}
where the constants $\tilde{C}_5,\tilde{C}_6,\tilde{C}_7,\tilde{C}_8$ only depend on the dimension $n$ and the constant in \eqref{eqn:quasi-isometric}.
\end{corollary}
\begin{proof}	
	With $(e_1,\ldots,e_n)$ a $g$-ONB consisting of Eigenvectors of $\mathscr A$ we have by the last lemma
	\begin{align*}
		|M_{h,g}(\sigma)| &= \left|\frac{1}{4} \sum_{i=1}^{n}e_i \cdot \widetilde{T}_{h,g}(\mathscr A^{-1/2}e_i) \cdot \sigma\right|\leq \frac 14 \sum_{i=1}^n |e_i| |\widetilde{T}_{h,g}(\mathscr A^{-1/2}e_i)\cdot \sigma|
		\leq \frac 14 |\mathscr A^{-1}|^{1/2} \sum_{i=1}^n |\widetilde{T}_{h,g}(e_i)\cdot \sigma|\\&\leq \frac C4 |\mathscr A^{-1}|^{1/2} |\widetilde{T}_{h,g}| |\sigma|\,,
	\end{align*}
	and similarly for $|M_{g,h}|$.
\end{proof}

\section{Bismut Derivative Formulae}\label{bismut}

In this section, we fix a geodesically complete metric $g$ on $M$, so that the dependence of the data on $g$ can be safely ommited in the notation. For the simplicity of the presentation we are going to assume that the Riemannian manifold $M\equiv (M,g)$ is stochastically complete, which means that for the integral kernel of the unique self-adjoint realization of the Laplace-Beltrami operator $\Delta\geq 0$ one has
$$
 \int_M\mathrm{e}^{-t\Delta}(x,y)\dd\mu(y)=1,\quad\text{rather than the generally valid}\quad \int_M\mathrm{e}^{-t\Delta}(x,y)\dd\mu(y)\leq 1.
$$ 
This assumption is satisfied, for example, if the Ricci curvature of $M$ is bounded from below by constant (this criterion relies on geodesic completeness). We denote the spinor bundle as above with $\Sigma M$. \vspace{1mm}
	
We now record the Driver-Thalmaier machinery for probabilistic derivative formulae for $P_t=e^{-tD^2}$, noting that we have collected (essentially) all probabilistic definitions that are used in the sequel in the appendix of this paper. \\	
Let $(\Omega, \IFF, \IFF_*, \IP)$ be a filtered probability space which satisfies the usual conditions and which for every $x\in M$ carries an adapted Brownian motion 
$$
\mathsf{b}(x):[0,\infty)\times\Omega\longrightarrow M
$$
in $M$ starting from $x\in M$ which is generated by $\Delta$ (rather than $\Delta/2$, the latter being the more common choice in probability). In other words, $\mathsf{b}(x)$ is adapted, has continuous paths and the transition density of $\mathsf{b}(x)$ with respect to $\mu$ is given by the heat kernel $(t,y)\mapsto \mathrm{e}^{-t\Delta}(x,y)$. In particular, in view of
$$
\IP\{\mathsf{b}_t(x)\in M\}=\int_M \mathrm{e}^{-t\Delta}(x,y) d\mu(y),\quad  t>0,
$$ 
stochastic completeness just means that Brownian paths have an infinite lifetime (while on a general Riemannian manifold Brownian motion takes values in the Alexandroff compactification). We denote with the usual abuse of notation with $\transport^x$ the stochastic parallel transport along $\mathsf{b}(x)$ with respect to \emph{any} metric connection; for any $t\geq 0$ it is a pathwise unitary map from the fiber over $x$ to the fiber over $\mathsf{b}_t(x)$.

\begin{theorem}[Covariant Feynman-Kac formula] Assume $M$ is geodesically complete with $\mathrm{Ric}\geq -C'$ for some constant $C'\geq 0$\footnote{so that $M$ is stochastically complete}. Then for all $t>0$, $\psi\in \Gamma_{C^{\infty}_c}(M,\Sigma M)$, $x\in M$ one has
$$
P_t\psi(x)= \int \mathrm{e}^{-\frac{1}{4}\int^t_0\mathrm{scal}(\mathsf{b}_s(x)) \dd s}\transport^{x,-1}_t \psi(\mathsf{b}_t(x)) \dd\IP.
$$ 
\end{theorem}

\begin{proof} That both sides agree $\mu$-almost everywhere is a well-known fact. To see that the RHS actually is the smooth representative of the semigroup generated by (the unique self-adjoint realization of) $D^2$, one can use the classical probabilistic argument from the compact case under the given assumptions: pick $C>0$ with $\mathrm{scal}\geq -C$. The process 
\begin{align*}
&N:[0,t]\times \Omega \longrightarrow (\Sigma M)_x,\\
&N_r:=\mathrm{e}^{-\frac{1}{4}\int^{r}_0\mathrm{scal}(\mathsf{b}_s(x)) \dd s}\transport^{x,-1}_r P_{t-r} \psi(\mathsf{b}_r(x))
\end{align*}
is a local martingale (Proposition 3.2 in \cite{driver}), where we remark that by local parabolic regularity and the smoothness of $\psi$ the time-dependent section $(s,y)\mapsto P_s\psi(y)$ is actually smooth on $[0,\infty)\times M$, that is, up to $s=0$. Using the Lichnerowicz formula
\begin{align}\label{eqn:Lichnerowicz}
	D^2=\widetilde{\nabla}^*\widetilde{\nabla}+\frac{1}{4}\mathrm{scal}
\end{align}one finds the following Kato-Simon inequality \cite{batu2}
\begin{align}\label{ks}
|P_s(x,y)|\leq  \mathrm{e}^{-\frac{sC}{4}} \mathrm{e}^{-s\Delta}(x,y)\quad\text{ for all $s>0$, $x,y\in M$,}
\end{align}
where 
$$
M\times M\ni (x,y)\longmapsto P_s(x,y)\in \mathrm{Hom}\big((\Sigma M)_y,(\Sigma M)_x\big)\subset \Sigma M\boxtimes (\Sigma M)^*
$$
denotes the integral kernel of $P_s$, so that
\begin{align*}
|P_s\psi(y)|\leq \mathrm{e}^{-s(\Delta+C/4)}|\psi|(y)\quad\text{ for all $s>0$, $y\in M$.}
\end{align*}
It follows that
\begin{align*}
|P_s\psi(y)|\leq \mathrm{e}^{\frac{sC}{4}}\mathrm{e}^{-s\Delta}|\psi|(y)=\mathrm{e}^{\frac{sC}{4}}\int_M\mathrm{e}^{-s\Delta}(y,y')|\psi|(y')\dd \mu(y')\leq \mathrm{e}^{\frac{sC}{4}}\left\|\Psi\right\|_{\infty}\int_M\mathrm{e}^{-s\Delta}(y,y')\dd\mu(y')\leq \mathrm{e}^{\frac{sC}{4}}\left\|\Psi\right\|_{\infty}.
\end{align*}
Thus, as $\transport_r$ is pathwise unitary, for all $r\in [0,t]$ we have
\begin{align*}
\mathrm{e}^{-\frac{1}{4}\int^{r}_0\mathrm{scal}(\mathsf{b}_s(x)) \dd s}\transport^{x,-1}_r P_{t-r} \psi(\mathsf{b}_r(x))\leq \mathrm{e}^{\frac{tC}{4}}|P_{t-r} \psi(\mathsf{b}_r(x))|\leq \mathrm{e}^{\frac{tC}{4}} \mathrm{e}^{\frac{tC}{4}}\left\|\Psi\right\|_{\infty},
\end{align*}
and so
\begin{align*}
\int \sup_{r\in [0,t]}|N_r|\dd\IP<\infty,
\end{align*}
so that $N$ is actually a true martingale (being a uniformly integrable local martingale) and thus its expectation value is constant in time. Thus,
$$
P_t\psi(x)=\int N_0 \dd\IP=\int N_t \dd\IP=\int \mathrm{e}^{-\frac{1}{4}\int^t_0\mathrm{scal}(\mathsf{b}_s(x)) \dd s}\transport^{x,-1}_t \psi(\mathsf{b}_t(x))\dd\IP,
$$
which completes the proof.
\end{proof}

The Ricci curvature is read as a section 
$$
\mathrm{Ric}\in\Gamma_{C^{\infty}}(M,\mathrm{End}(TM)),
$$
and 
$$
\mathrm{Ric}'\in\Gamma_{C^{\infty}}(M,\mathrm{End}(T^*M))
$$
is defined by duality. Let
$$
R \in \Gamma_{C^{\infty}} (M,  T^*M\otimes T^*M  \otimes \mathrm{End}(T M) )
$$
denote the curvature tensor of the Levi-Civita connection and let
$$
\widetilde{R} \in \Gamma_{C^{\infty}} (M,  T^*M\otimes T^*M  \otimes \mathrm{End}(\Sigma M) )
$$
denote the curvature tensor of the Levi-Civita connection acting on spinors. The section
$$
\underline{\widetilde{R}}   \in \Gamma_{C^{\infty}}(M,\mathrm{End}(T^*M\otimes \Sigma M))
$$
is defined on $\phi\in (T^*M\otimes \Sigma M)_x$, $v\in T_xM$, and an orthonormal frame $e_1,\dots, e_n$ for $T_xM$ by 
$$
\underline{\widetilde{R}}(\phi)(v):=(\mathrm{Ric}'\otimes 1_{\Sigma M})(\phi)(v)-2\sum^n_{j=1}  \widetilde{R}(v,e_i)\phi(v)+\frac{1}{4}\mathrm{scal} \cdot \phi(v),
$$
and the section
$$
\rho   \in \Gamma_{C^{\infty}}(M,\mathrm{Hom}(\Sigma M, T^*M\otimes \Sigma M))
$$
is defined on $\psi\in (\Sigma M)_x$ by
	\begin{align*}
    	\rho (\psi)(v) = \frac{1}{4}(\mathrm{grad}(\mathrm{scal}),v)\psi +\sum^n_{j=1}(\widetilde{\nabla}_{e_i}\widetilde{R})(e_i,v)\psi.
	\end{align*}

We define a continuous adapted process by
\begin{align*}
&\underline{Q}(x) :[0,\infty)\times \Omega\longrightarrow\mathrm{End}\big((T^*M\otimes \Sigma M)_x\big ),\\
(\dd  / \dd s)\underline{Q}_s(x)  &= -  \underline{Q}_s(x) \big( \transport_s^{x,-1} \underline{\widetilde{R}}(\mathsf{b}_s(x))\transport^x_s \big),\quad    \underline{Q}_0(x)  = 1.
    \end{align*}
Let
$$
\underline{\mathsf{b}}(x):[0,\infty)\times \Omega\longrightarrow T_{x}M
$$
denote the stochastic anti-development of $\mathsf{b}(x)$. The actual definition of $\underline{\mathsf{b}}$ (cf. \cite{hsu}) will play no role for us; it will only be essential to know that this process is an adapted Euclidean Brownian motion in $T_xM$. For every $r>0$ let
$$
\tau(x,r) :=\inf\{t\geq 0: \mathsf{b}_{t}(x)  \notin B(x,r)\}: \Omega \longrightarrow [0,\infty]
$$
be the first exit time of $\mathsf{b}(x) $ from the open ball $B(x,r)$. \\
Given $r>0$, $t>0$, $x\in M$, $v\in (TM^*\otimes \Sigma M)_x$ define a set of processes $\mathscr{P}_1(x,r,t,v)$ to be given by all adapted processes having absolutely continuous paths
$$
\ell: [0,t] \times \Omega\longrightarrow (T^*M\otimes \Sigma M)_x,
$$
such that   
$$
\int\int_0^{t\wedge\tau(x,r) }\abs{\dot\ell_s}^2 \dd s\dd\IP<\infty,\quad  \ell_0 =v,\quad \ell_s = 0\quad \text{for all $s \geq t \wedge \tau(x,r)$}.
$$
For $\ell \in \mathscr{P}_1(x,r,t,v)$ define a continuous adapted process by
\begin{align*}
   & U^{\ell} :[0,t]\times \Omega\longrightarrow (\Sigma M)_x ,\\
		& U^{\ell}_r:=  \int_0^{r} \mathrm{e}^{\int^s_0\frac{1}{4}\mathrm{scal}(\mathsf{b}_u(x))\dd u} \mathscr{G}( \dd \underline{\mathsf{b}}_s(x)) \underline{Q}^*_s(x) \dot\ell_s+\frac{1}{2} \int_0^{r}  \mathrm{e}^{\int^s_0\frac{1}{4}\mathrm{scal}(\mathsf{b}_u(x))\dd u}   \transport_s^{x,-1}  \rho(\mathsf{b}_s(x))^* \transport^x_s \underline{Q}^*_s(x) \ell_s  \dd s,
			\end{align*}
where 
$$
 \mathscr{G}\in \Hom\Big(T_xM,\Hom\big((T^*M\otimes \Sigma M)_x
,(\Sigma M)_x\big)\Big)
$$
is given by
$$
 \mathscr{G}(v)A:= Av,\quad A\in (T^*M\otimes \Sigma M)_x=\Hom(T_xM,(\Sigma M)_x),\quad v\in T_x M.
$$

\begin{theorem}[Bismut derivative formula I]\label{B1} Assume $M$ is geodesically complete with $\mathrm{scal}\geq -C$ for some constant $C>0$. Then for all $x\in M$, $t,r>0$ $\psi\in\Gamma_{C^{\infty}_c}(M,\Sigma M)$, $v\in (T^*M\otimes \Sigma M)_x$, $\ell \in \mathscr{P}_1(x,r,t,v)$ one has
$$
\left(\widetilde{\nabla}P_t \psi(x), v\right)= -  \int \mathrm{e}^{-\frac{1}{4}\int^t_0\mathrm{scal}(\mathsf{b}_s(x)) \dd s}\left(\transport_t^{x,-1} \psi(\mathsf{b}_t(x)),U^{\ell}_{t\wedge \tau(x,r)} \right)\dd\IP.
$$  
\end{theorem}

\begin{proof} With $N$ as in the above proof and
\begin{align*}
&\underline{N}:[0,t]\times \Omega \longrightarrow (T^*M\otimes\Sigma M)_x,\\
&\underline{N}_r:=\mathrm{e}^{-\frac{1}{4}\int^{r}_0\mathrm{scal}(\mathsf{b}_s(x)) \dd s}\transport^{x,-1}_r \widetilde{\nabla}P_{t-r} \psi(\mathsf{b}_r(x))
\end{align*}
the process 
$$
Z:[0,t]\times \Omega \longrightarrow \IC,\quad Z_r:= (\underline{N}_r, \ell_r)- (N_r,U_r^{\ell})
$$
is a local martingale (Theorem 3.7 in \cite{driver}). It follows that the stopped process $Z^{ \tau(x,r) }$ is a martingale (as it is a uniformly integrable local martingale by the Burkholder-Davis-Gundy inequality; cf.\ (\ref{bdg}) in the appendix) so that 
\begin{align*}
&\left(\widetilde{\nabla}P_t \psi(x), v\right)=\int Z_0 \dd \IP=\int Z_{t\wedge \tau(x,r)} \dd \IP \\
&= -\int\left( \mathrm{e}^{-\frac{1}{4}\int^{t\wedge\tau(x,r)}_0\mathrm{scal}(\mathsf{b}_s(x)) \dd s}\transport_{t\wedge\tau(x,r)}^{x,-1} P_{t-t\wedge\tau(x,r)}\psi(\mathsf{b}_{t\wedge\tau(x,r)}(x)) , U^{\ell}_{t\wedge\tau(x,r)}\right)\dd \IP.
\end{align*}
Using the covariant Feynman-Kac formula and the strong Markoff property of Brownian motion, the RHS of the latter equation is precisely the RHS of the first Bismut derivative formula. 
\end{proof}

Given $r>0$, $t>0$, $x\in M$, $\zeta\in  (\Sigma M)_x$ define a set of processes $\mathscr{P}_2(x,r,t,\zeta)$ to be given by all pathwise absolutely continuous processes 
$$
\ell: [0,t] \times \Omega\longrightarrow (\Sigma M)_x
$$
such that   
$$
\int\int_0^{t\wedge\tau(x,r) }\abs{\dot\ell_s}^2 \dd s\dd \IP<\infty,\quad  \ell_0 =\zeta,\quad \ell_s = 0\quad \text{for all $s \geq t \wedge \tau(x,r)$}.
$$

\begin{theorem}[Bismut derivative formula II]\label{B2} Assume $M$ is geodesically complete with $\mathrm{Ric}\geq -C$ for some constant $C>0$. Then for all $x\in M$, $t,r>0$ $\psi\in\Gamma_{C^{\infty}_c}(M,\Sigma M)$, $\zeta\in ( \Sigma M)_x$, $\ell \in \mathscr{P}_2(x,r,t,\zeta)$ one has
$$
\left(DP_t \psi(x), \zeta\right)=   \int \mathrm{e}^{-\frac{1}{4}\int^t_0\mathrm{scal}(\mathsf{b}_s(x)) \dd s}\left(\transport_t^{x,-1} \psi(\mathsf{b}_t(x)),\int^{t\wedge \tau(x,r)}_0 \dd\underline{\mathsf{b}}_s(x) \cdot \dot{\ell}_s \right)\dd\IP.
$$  
\end{theorem}

\begin{proof} With $N$ as in the above proof and
\begin{align*}
&\underline{N}:[0,t]\times \Omega \longrightarrow \Sigma M_x,\\
&\underline{N}_r:=\mathrm{e}^{-\frac{1}{4}\int^{r}_0\mathrm{scal}(\mathsf{b}_s(x)) \dd s}\transport^{x,-1}_r DP_{t-r} \psi(\mathsf{b}_r(x))
\end{align*}
the process 
$$
Z:[0,t]\times \Omega \longrightarrow \IC,\quad Z_r:= (\underline{N}_r, \ell_r)+ \left(N_r,\int^{r}_0 \dd\underline{\mathsf{b}}_s(x) \cdot \dot{\ell}_s\right)
$$
is a local martingale (again by Theorem 3.7 in \cite{driver}). It follows that the stopped process $Z^{ \tau(x,r) }$ is a martingale so that 
\begin{align*}
&\left(DP_t \psi(x),\zeta\right)=\int Z_0 \dd \IP=\int Z_{t\wedge \tau(x,r)} \dd \IP \\
&= \int\left( \mathrm{e}^{-\frac{1}{4}\int^{t\wedge\tau(x,r)}_0\mathrm{scal}(\mathsf{b}_s(x)) \dd s}\transport_{t\wedge\tau(x,r)}^{x,-1} P_{t-t\wedge\tau(x,r)}\psi(\mathsf{b}_{t\wedge\tau(x,r)}(x)) ,\int^{t\wedge \tau(x,r)}_0 \dd\underline{\mathsf{b}}_s(x) \cdot \dot{\ell}_s\right) \dd \IP.
\end{align*}
Using again the covariant Feynman-Kac formula and the strong Markoff property of Brownian motion, the RHS of the latter equation is precisely the RHS of the second Bismut derivative formula. 
\end{proof}

We record some consequences of the Feynman-Kac formula and the Bismut derivative formula, respectively: to this end, let $\ILL(\IHH_1,\IHH_2)$denote the space of bounded operators between two Hilbert spaces $\IHH_1$, $\IHH_2$, and for $p\in[1,\infty)$ let $\ILL^p(\IHH_1,\IHH_2)$ denote the $p$-th Schatten class (so $p=1$ is the trace class and $p=2$ is the Hilbert-Schmidt class), where
$$
\ILL(\IHH_1):=\ILL(\IHH_1,\IHH_1),\quad \ILL^p(\IHH_1):=\ILL^p(\IHH_1,\IHH_1).
$$

\begin{remark}\label{abst}
\begin{align}
&T\in \ILL^p \Leftrightarrow T^*\in \ILL^p\quad\text{for all $p$},\\
&\ILL^p\circ \ILL\subset \ILL^p\quad\text{for all $p$},\\
&\ILL^p\subset \ILL^q\quad\text{for all $p\leq q$},\\
&\ILL^p\circ \ILL^q\subset \ILL^r\quad\text{for all $p,q,r$ with $1/p+1/q=1/r$}.
\end{align}
\end{remark}

In particular, the product of two Hilbert-Schmidt operators is trace class, and the product of a bounded operator and a trace class operator (resp. Hilbert-Schmidt operator) is again trace class (resp. Hilbert-Schmidt). Assume we are given metric vector bundles $E,F$ over $M$ and a bounded operator 
$$
T\in \ILL\big(\Gamma_{L^2}(M,E),\Gamma_{L^2}(M,F)\big)
$$
which is given by a pointwise well-defined $L^1_\loc$-integral kernel, that is, 
$$
Tf(x)=\int_M T(x,y)f(y) \dd\mu(y),
$$ 
where 
$$
T(\cdot,\cdot)\in \Gamma_{L^1_{\loc}}(M\times M, E\boxtimes F^*)
$$
that is, an $L^1_\loc$-map
$$
M\times M\ni (x,y)\longmapsto T(x,y)\in\mathrm{Hom}(E_y,F_x)\in E\boxtimes F^*.
$$
Then one has
$$
T\in\ILL^2(\Gamma_{L^2}(M,E),\Gamma_{L^2}(M,F)),\quad\text{ if $\>\int_M\int_M |T(x,y)|^2 \dd\mu(x)\dd\mu(y)<\infty$}.
$$

\begin{corollary}\label{aappq0} Assume $M$ is geodesically complete with $\mathrm{Ric}\geq -C$ for some constant $C> 0$ and let $t>0$. Then for every metric vector bundle $E$ over $M$, and every 
$$
A\in \Gamma_{L^1_{\loc}}\big(M,\mathrm{Hom}(\Sigma M,E)\big)\quad \text{(considered as a multiplication operator)}
$$
which satisfies
$$
\int_M\frac{|A(x)|^2}{  \mu(B(x,\sqrt{t}))} \dd\mu(x)<\infty,\quad\text{one has}\quad AP_t\in \ILL^2\big(\Gamma_{L^2}(M,\Sigma M),\Gamma_{L^2}(M,E)\big).
$$
\end{corollary}

\begin{proof} We note that by the Li-Yau heat kernel estimate for the scalar Laplacian one has
$$
\mathrm{e}^{-t\Delta}(x,y)\leq \frac{C_t}{\mu(B(x,\sqrt{t}))} \quad\text{ for all $t>0$, $x,y\in M$},
$$
for some $C_t<\infty$, where in the sequel, $C_t$ denotes a constant which depends on $t$ but is uniform in $x,y\in M$, whose actual value may differ from line to line. By the Kato-Simon inequality (\ref{ks}) we have 
\begin{align*}
|[AP_t](x,y)|=|A(x)P_t(x,y)|\leq C_t |A(x)| \mathrm{e}^{-t\Delta}(x,y),
\end{align*}
so that by the Li-Yau estimate 
\begin{align*}
&\int_M\int_M |[AP_t](x,y)|^2 \dd\mu(x)\dd\mu(y)\leq C_t\int_M|A(x)|^2\int_M \mathrm{e}^{-t\Delta}(x,y)^2 \dd\mu(y)\dd\mu(x)\\
&=C_t\int_M|A(x)|^2 \mathrm{e}^{-t\Delta}(x,x) \dd\mu(x)\leq C_t\int_M\frac{|A(x)|^2}{  \mu(B(x,\sqrt{t}))} \dd\mu(x),
\end{align*}
completing the proof.
\end{proof}

Now we can prove:

\begin{corollary}\label{aappq1} Assume $M$ is geodesically complete with $|R|\leq C$ for some $C<\infty$ and let $t>0$. Then for every metric vector bundle $E$ over $M$, every
$$
A\in \Gamma_{L^1_{\loc}}\big(M,\mathrm{Hom}(T^*M\otimes \Sigma M,E)\big)
$$
which satisfies  
$$
\int_M\frac{\big(1+\max_{y\in B(x,1)}|\rho(y)|\big)^2|A(x)|^2}{\mu(B(x,\sqrt{t}))} \dd \mu(x)<\infty,\quad\text{ one has }\quad \>A\widetilde{\nabla}P_t\in \ILL^2\Big(\Gamma_{L^2}(M,\Sigma), \Gamma_{L^2}(M,E)\Big).
$$
\end{corollary}

\begin{proof} Note that by $|R|\leq C$, the tensor $\underline{\widetilde{R}}$ as well as the Ricci and the scalar curvature are bounded. In the sequel, $C$ and $C_t$ denote constants whose actual value can change from line to line. Let $x\in M$, $\psi\in\Gamma_{C^{\infty}_c}(M,\Sigma M)$, $v\in (T^*M\otimes \Sigma M)_x$, $\ell \in \mathscr{P}_1(x,1,t,v)$ so that by the first Bismut derivative formula and $\mathrm{scal}\geq -C$ we have
\begin{align*}
&\left|\left(\widetilde{\nabla}P_t \psi(x), v\right)\right|\\
&\leq  C_t\int \left|\psi(\mathsf{b}_t(x))\right| \left|\int_0^{t\wedge \tau(x,1)} \mathrm{e}^{-\int^s_0\frac{1}{4}\mathrm{scal}(\mathsf{b}_u(x))\dd u} \dd \underline{\mathsf{b}}_s(x) \underline{Q}^*_s(x) \dot\ell_s  \right|\dd\IP\\
&+C_t\int \left|\psi(\mathsf{b}_t(x))\right|  \int_0^{t\wedge \tau(x,1)}  \mathrm{e}^{-\int^s_0\frac{1}{4}\mathrm{scal}(\mathsf{b}_u(x))\dd u} \left|  \transport_s^{x,-1}  \rho(\mathsf{b}_s(x))^* \transport^x_s \underline{Q}^*_s(x) \ell_s   \right|\dd s\dd\IP,
\end{align*} 
so that by Cauchy-Schwarz the latter is
\begin{align*}
&\leq  C_t\left(\int \left|\psi(\mathsf{b}_t(x))\right|^2 \dd\IP\right)^{1/2} \left(\int\left|\int_0^{t\wedge \tau(x,1)} \mathrm{e}^{-\int^s_0\frac{1}{4}\mathrm{scal}(\mathsf{b}_u(x))\dd u} \dd \underline{\mathsf{b}}_s(x) \underline{Q}^*_s(x) \dot\ell_s  \right|^2\dd\IP\right)^{1/2}\\
&+C_t\left(\int \left|\psi(\mathsf{b}_t(x))\right|^2 \dd\IP\right)^{1/2} \left(\int \left(\int_0^{t\wedge \tau(x,1)}  \mathrm{e}^{-\int^s_0\frac{1}{4}\mathrm{scal}(\mathsf{b}_u(x))\dd u}  \left| \transport_s^{x,-1}  \rho(\mathsf{b}_s(x))^* \transport^x_s \underline{Q}^*_s(x) \ell_s \right| \dd s\right)^2\dd\IP\right)^{1/2}\,.
\end{align*} 

We use $\mathrm{scal}\geq -C$, that by Gronwall's inequality 
$$
\left|\underline{Q}^*_s(x)\right|\leq C_t\quad \text{ for all $0\leq s\leq t$},
$$
and 
$$
|\rho(\mathsf{b}_s(x))^*|\leq  \max_{y\in B(x,1)}|\rho(y)| \quad \text{ for all $0\leq s\leq \tau(x,1)$}
$$
for the second integral, and the Burkholder-Davis-Gundy inequality to bound the first integral, so that the above is
\begin{align*}
&\leq  C_t\left(\int \left|\psi(\mathsf{b}_t(x))\right|^2 \dd\IP\right)^{1/2} \left(\int\int_0^{t\wedge \tau(x,1)} \mathrm{e}^{-\int^s_0\frac{1}{2}\mathrm{scal}(\mathsf{b}_u(x))\dd u}  \left|\underline{Q}^*_s(x)\right|^2 \left|\dot\ell_s  \right|^2\dd s\dd\IP\right)^{1/2}\\
&+C_t  \max_{y\in B(x,1)}|\rho(y)|\left(\int \left|\psi(\mathsf{b}_t(x))\right|^2 \dd\IP\right)^{1/2} \left(\int \left(\int_0^{t}  \left|\ell_s   \right|\dd s\right)^2\dd\IP\right)^{1/2},
\end{align*} 
which is
\begin{align*}
&\leq  C_t\left(\int \left|\psi(\mathsf{b}_t(x))\right|^2 \dd\IP\right)^{1/2} \left(\int\int_0^{t\wedge \tau(x,1)} \left|\dot\ell_s  \right|^2\dd s\dd\IP\right)^{1/2}\\
&+C_t\max_{y\in B(x,1)}|\rho(y)|\left(\int \left|\psi(\mathsf{b}_t(x))\right|^2 \dd\IP\right)^{1/2} \left(\int \left(\int_0^{t}  \left|\ell_s   \right|\dd s\right)^2\dd\IP\right)^{1/2}.
\end{align*}
Now, using a lower bound for the Ricci curvature, $\ell$ can be chosen (cf. the proof of Corollary 5.1 in \cite{tw}) such that 
\begin{align*}
|\ell|\leq |v|,\quad \left(\int\int_0^{t\wedge \tau(x,1)} \left|\dot\ell_s  \right|^2\dd s\dd\IP\right)^{1/2}\leq C_t|v|, 
\end{align*}
so that we arrive at 
\begin{align*}
&\left|\left(\widetilde{\nabla}P_t \psi(x), v\right)\right|\leq  C_t(1+\max_{y\in B(x,1)}|\rho(y)|)\left(\int \left|\psi(\mathsf{b}_t(x))\right|^2 \dd\IP\right)^{1/2} |v|\\
&=C_t(1+\max_{y\in B(x,1)}|\rho(\mathsf{b}_s(y))|)\left(\int_M \mathrm{e}^{-t\Delta}(x,y)|\psi(y)|^2\dd \mu(y)\right)^{1/2}|v|\\
&\leq \frac{C_t(1+\max_{y\in B(x,1)}|\rho(y)|)}{\sqrt{\mu(B(x,\sqrt{t}))}}\left(\int_M |\psi(y)|^2\dd \mu(y)\right)^{1/2}|v|=\frac{C_t(1+\max_{y\in B(x,1)}|\rho(\mathsf{b}_s(y))|)}{\sqrt{\mu(B(x,\sqrt{t}))}}\left\|\psi\right\|_2|v|.
\end{align*}
Using Riesz-Fischer's duality theorem this estimate implies
$$
\int_M\left|[\widetilde{\nabla}P_t](x,y)\right|^2 \dd \mu(y)\leq \frac{C_t(1+\max_{y\in B(x,1)}|\rho(y)|)^2}{\mu(B(x,\sqrt{t}))},
$$
so that
\begin{align*}
&\int_M\int_M\left|[A\widetilde{\nabla}P_t](x,y)\right|^2 \dd \mu(y) \dd\mu(x)=\int_M\int_M\left|A(x)[\widetilde{\nabla}P_t](x,y)\right|^2 \dd \mu(y) \dd \mu(x)\\
&\leq C_t \int_M\frac{(1+\max_{y\in B(x,1)}|\rho(y)|)^2|A(x)|^2}{\mu(B(x,\sqrt{t}))} \dd \mu(x),
\end{align*}
which completes the proof.
\end{proof}

\begin{corollary}\label{aappq2} Assume $M$ is geodesically complete with $\mathrm{scal}\geq -C$ for some $C>0$ and let $t>0$. Then for every  metric vector bundle $E$ over $M$, every
$$
A\in \Gamma_{L^1_\loc}\big(M,\mathrm{Hom}( \Sigma M,E)\big)
$$
which satisfies  
$$
\int_M\frac{|A(x)|^2}{\mu(B(x,\sqrt{t}))} \dd \mu(x)<\infty,\quad\text{ one has $\>ADP_t\in \ILL^2\Big(\Gamma_{L^2}(M,\Sigma M), \Gamma_{L^2}(M,E)\Big)$}.
$$
\end{corollary}

\begin{proof} Let $x\in M$, $\psi\in\Gamma_{C^{\infty}_c}(M,\Sigma M)$, $\zeta\in (\Sigma M)_x$, $\ell \in \mathscr{P}_2(x,1,t,\zeta)$ so that by the second Bismut derivative formula, $\mathrm{scal}\geq -C$, Cauchy-Schwarz and Burgholder-Davis-Gundy we have
\begin{align*}
\left|\left(DP_t \psi(x), \zeta\right)\right|\leq C_t\left(\int \left|\psi(\mathsf{b}_t(x))\right|^2 \dd\IP\right)^{1/2} \left(\int \int_0^{t\wedge \tau(x,1)}   \left|  \dot{\ell}_s \right|^2 \dd s\dd\IP\right)^{1/2}.
\end{align*} 
Now we choose $\ell$ such that (cf. again the proof of Corollary 5.1 in \cite{tw})
\begin{align*}
|\ell|\leq |\zeta|,\quad \left(\int\int_0^{t\wedge \tau(x,1)} \left|\dot\ell_s  \right|^2\dd s\dd\IP\right)^{1/2}\leq C_t|\zeta|, 
\end{align*}
so that we arrive at 
\begin{align*}
\left|\left(DP_t \psi(x), \zeta\right)\right|\leq C_t\left(\int \left|\psi(\mathsf{b}_t(x))\right|^2 \dd\IP\right)^{1/2} |\zeta|.
\end{align*}
From here on one can copy the proof of Corollary \ref{aappq1}.
\end{proof}

\section{Main Result}\label{mainres}

Assume $g$ and $h$ are geodesically complete Riemannian metrics on $M$ and denote by $\mathcal Q_j$ the nonnegative closed sesquilinear form corresponding to $D^2_j$, i.e., $\mathcal Q_j(\psi) = \langle D_j^2\psi,\psi\rangle=\|D_j\psi\|^2$ with $\Dom(\mathcal Q_j)=\Dom(D_j)=\Dom(\sqrt{D_j^2})$.

\begin{lemma}\label{lem:quadratic-form-domains}
		If $g\sim h$ are geodesically complete Riemannian metrics on $M$ with bounded scalar curvatures and such that the function $\omega_{g,h}$ is bounded, then
		\[
			I_{g,h}\Dom(\mathcal Q_g) = \Dom(\mathcal Q_h)\,.
		\]
\end{lemma}
\begin{proof}Throughout the proof, $C$ is a positive constant whose value might change from line to line, but whose existence is guaranteed by the assumptions of the lemma.
	
	The space $\Dom(\mathcal Q_j)$ is the closure of $\Gamma_{C^\infty_c}(M,\Sigma_jM)$ with respect to the graph norm
	\[
		\psi \mapsto \left(\|\psi\|^2 + \|D_j\psi\|^2\right)^{1/2}\,.
	\]
	By the Lichnerowicz formula \eqref{eqn:Lichnerowicz} we have for $\psi\in\Gamma_{C^\infty_c}(M,\Sigma_gM)$
	\begin{align*}
		\|D_hI\psi\|^2 &= \left\langle D_h^2I\psi,I\psi \right\rangle=  \left\langle \left({\widetilde{\nabla}}^{h*}\widetilde{\nabla}^h+\tfrac{1}{4}\mathrm{scal}_h\right)I\psi,I\psi \right\rangle \\
		&=\int_{M}\!|\widetilde{\nabla}^h\beta^g_h\psi(x)|^2\dd\mu_h(x) + \frac 14 \int_{M}\!(\textrm{scal}_h(x)\beta^g_h\psi(x),\beta^g_h\psi(x))\dd\mu_h(x)\,.
	\end{align*}
	Since we assumed quasi-isometry of the metrics and boundedness of the scalar curvatures, the second term is bounded by $C\|\psi\|^2$.
	
	By Remark~\ref{rem:skewed-connection}(ii), Lemma~\ref{lem:tilde-T-bounds} and the assumption on $\omega_{g,h}$ we can bound the first term as follows,
	\begin{multline*}
		\int_{M}\!|\widetilde{\nabla}^h\beta^g_h\psi(x)|^2\dd\mu_h(x) =
		\int_{M}\!|\beta^g_h\beta^h_g\widetilde{\nabla}^h\beta^g_h\psi(x)|^2\dd\mu_h(x) = \int_{M}\!|{}^g\widetilde{\nabla}^h\psi(x)|^2\dd\mu_h(x) = \int_{M}\!|(\widetilde{\nabla}^g+\tfrac 14\widetilde{T}_{h,g})\psi(x)|^2\dd\mu_h(x)\\
		\leq \int_{M}\!\left(|\widetilde{\nabla}^g\psi(x)|+\tfrac 14|\widetilde{T}_{h,g}\psi(x)|\right)^2\varrho_{g,h}(x)\dd\mu_g(x)\leq C \left(\|\widetilde{\nabla}^g\psi\|^2 + \|\omega_{g,h}\|^2_\infty\|\psi\|^2 + \|\widetilde{\nabla}^g\psi\|\|\omega_{h,g}\|_\infty\|\psi\|\right)\\
		\leq C \left(\|\widetilde{\nabla}^g\psi\|^2 + \|\psi\|^2 + \|\widetilde{\nabla}^g\psi\|\|\psi\|\right)
	\end{multline*}
We use the Lichnerowicz formula once more to obtain
	\begin{align*}
		\|\widetilde{\nabla}^g\psi\|^2 &= \left\langle \widetilde{\nabla}^{g*}\widetilde{\nabla}^g\psi,\psi\right\rangle = \left\langle (\widetilde{\nabla}^{g*}\widetilde{\nabla}^g+\tfrac14\textrm{scal}_g)\psi,\psi\right\rangle - \tfrac14 \left\langle \textrm{scal}_g\psi,\psi \right\rangle\leq (D_g^2\psi,\psi) + C\left\langle \psi,\psi \right\rangle\\
		&\leq C (\|D_g\psi\|^2+\|\psi\|^2)\,.
	\end{align*}
Putting everything together, we have
	\begin{align*}
	\|I\psi\|^2+\|D_hI\psi\|^2 \leq C \Big(\|D_g\psi\|^2 + \|\psi_g\|^2 + \sqrt{\|D_g\psi\|^2 + \|\psi_g\|^2}\cdot\|\psi\|\Big)\,,
	\end{align*}
	which implies
	\begin{align}\label{eqn:lem:quadratic-form-domains-0}
		I\Dom(\mathcal Q_g)\subseteq \Dom(\mathcal Q_h)\,.
	\end{align}
	Because of $I^{-1}=I_{g,h}^{-1}=I_{h,g}$ and since our arguments are symmetric in $g$ and $h$, we have equality in \eqref{eqn:lem:quadratic-form-domains-0}.
\end{proof}

Assuming $g$ and $h$ are geodesically complete Riemannian metrics on $M$, set
\begin{align*}
\Psi^{(1)}_{g,h}(x)&:=\max\big(\delta_{g,h}(x)^2,\omega_{g,h}(x)^2, \delta_{g,h}(x)\Psi_g(x)\big)\,,\\
\Psi^{(2)}_{g,h}(x)&:=\max\big(\omega_{g,h}(x), \delta_{g,h}(x)\Psi_g(x),\delta_{g,h}(x)\Psi_h(x)\big)\,,\\
\end{align*}
where
\begin{align*}
&\Psi_g:M\longrightarrow \IR, \quad \Psi_g(x):=\big(1+\max_{y\in B_g(x,1)}|\nabla^gR_g(y)|\big)^2,\\
&\Psi_h:M\longrightarrow \IR, \quad \Psi_h(x):=\big(1+\max_{y\in B_h(x,1)}|\nabla^hR_h(y)|\big)^2.
\end{align*}

Now we can prove our main result, refering the reader to section \ref{saab} for the basic notions and notations from scattering theory used here:

\begin{theorem}\label{main} Assume $g\sim h$ are geodesically complete Riemannian metrics on $M$ such that there exists a constant $C<\infty$ with 
$$
|\omega_{g,h}|+|R_g|+|R_h|\leq C\,.
$$
Then the following results hold true:\\
a) If for some $t>0$ and some (and then by $g\sim h$ both) $j\in \{g,h\}$ one has
\[
	\int_M \frac{\Psi^{(1)}_{g,h}(x)}{\mu_j(B_j(x,\sqrt{t}))} \dd \mu_j(x)<\infty\,,
\] 
then the wave-operators $\mathscr{W}_\pm(D_h,D_g,I_{g,h})$ exist and are complete. Moreover, the $\mathscr{W}_\pm(D_h,D_g,I_{g,h})$ are partial isometries with initial space $\Gamma_{L^2}(M,\Sigma_gM)^{\mathrm{ac}}(D_g) $ and final space $\Gamma_{L^2}(M,\Sigma_hM)^{\mathrm{ac}}(D_h) $. In particular, we have $\spec_{\mathrm{ac}}(D_g)=\spec_{\mathrm{ac}}(D_h)$.\\
b) If for some $t>0$ and some (and then by $g\sim h$ both) $j\in \{g,h\}$ one has
\[
	\int_M \frac{\Psi^{(2)}_{g,h}(x)}{\mu_j(B_j(x,\sqrt{t}))} \dd \mu_j(x)<\infty\,,
\] 
then the wave-operators $\mathscr{W}_\pm(D^2_h,D^2_g,I_{g,h})$ exist and are complete. Moreover, the $\mathscr{W}_\pm(D^2_h,D^2_g,I_{g,h})$ are partial isometries with initial space $\Gamma_{L^2}(M,\Sigma_gM)^{\mathrm{ac}}(D_g^2) $ and final space $\Gamma_{L^2}(M,\Sigma_h M)^{\mathrm{ac}}(D_h^2)$. In particular, we have $\spec_{\mathrm{ac}}(D^2_g)=\spec_{\mathrm{ac}}(D^2_h)$.
\end{theorem}

\begin{proof} We will use the Belopol'skii-Birman-Theorem (cf.\ Appendix \ref{saab}). We are only going to prove b), noting that the proof of a) is similar and in fact easier, using the second part of the Belopol'skii-Birman-Theorem.\\
Note that for all $x\in M$ one has
\begin{align*}
&\max_{y\in B_g(x,1)}|\rho_g(y)|\leq C'''\max_{y\in B_g(x,1)}|\nabla^gR_g(y)|,\\
& \max_{y\in B_h(x,1)}|\rho_h(y)|\leq C'''\max_{y\in B_h(x,1)}|\nabla^hR_h(y)|.
\end{align*}
Now Corollary \ref{aappq0}, Corollary \ref{aappq1}, Corollary \ref{aappq2}, Remark \ref{abst}, Corollary \ref{cor:M-bounds}, and Lemma \ref{lem:S-Shat-pointwise-estimates} imply that the operator $\mathscr{T}_{g,h,t}$ from the Dirac HPW-formula is trace class. Moreover, we have
$$
(I^*I-1)\exp(-tD_g^2)=I^{-1}U_{g,h}S_{g,h;g}S_{g,h;g}\exp(-tD_g^2),
$$
which by Corollary \ref{aappq0}, Remark \ref{abst} and Lemma \ref{lem:S-Shat-pointwise-estimates} is Hilbert-Schmidt. These facts together with Lemma \ref{lem:quadratic-form-domains} show that the assumptions of the first part of the Belopol'skii-Birman-Theorem are satisfied.
\end{proof}

\begin{example}
Let $(M^n,g)=(\mathbb{H}^n,\textrm{hyp})$ denote the $n$-dimensional hyperbolic space with constant sectional curvature $-1$. Then, being retractable, $\mathbb{H}^n$ carries precisely one spin structure. The spectrum of $D_g^2$ is given by
\[
\spec_{\mathrm{ac}}(D_g^2) = [0,\infty),\quad \spec_{\mathrm{sc}}(D_g^2) = \emptyset, \quad \spec_{\mathrm{pp}}(D^2) = \emptyset\,,
\] see \cite[Corollary~4.6]{bunke} and also \cite[p. 441 \& p. 456]{baer} for a minor correction. The spectrum of $D_g$ is symmetric about $0$ (this is true on any Riemannian spin manifold of dimension $n\not\equiv 3\pmod 4$, and can be deduced in the above situation from the symmetric space structure of $(\mathbb{H}^n,\textrm{hyp})$ for all $n$) which implies by the spectral mapping theorem that
\[
\spec(D_g) = \R\,.
\]
Regarding the individual parts of the spectrum, we can easily exclude that $D_g$ has eigenvalues because $D_g^2$ has none, i.e., 
\[
\spec_{\mathrm{pp}}(D_g) = \emptyset\,.
\]
Next, by working directly with the definition of spectral measures and absolute continuity, it is easy to show that\footnote{This would still hold true if we replaced the function $x\mapsto x^2$ by any measurable function $f:\R\to \R$ for which $f(N)$ and $f^{-1}(N)$ have Lebesgue-measure zero as soon as $N\subseteq \R$ itself has Lebesgue-measure zero.  A class of such functions is, e.g., the set of all $f\in C^1(\R)$ with $ \{ x \,|\, f'(x)=0 \} $ discrete.}
\[
\Gamma_{L^2}(M,\Sigma_gM)^{\mathrm{ac}}(D_g) = \Gamma_{L^2}(M,\Sigma_gM)^{\mathrm{ac}}(D_g^2)\,.
\]
From the above, we already know that $\Gamma_{L^2}(M,\Sigma_gM)^{\mathrm{ac}}(D_g^2) = \Gamma_{L^2}(M,\Sigma_gM)$ which implies that
\[
\spec_{\mathrm{ac}}(D_g) = \R \quad\text{ and }\quad \spec_{\mathrm{sc}}(D_g) = \emptyset\,.
\]

Because $(\mathbb{H}^n,\textrm{hyp})$ is homogeneous, the volume $\mu_g(B(x,\sqrt{t}))$ does not depend on the point $x$. Now let $h$ be any Riemannian metric on $\mathbb{H}^n$ with $g\sim h$ and such that there exists $C<\infty$ with
\[
|\omega_{g,h}|	+ |R_h| \leq C\,.
\]
Then we have:\\
a) If
\[
\int_M \Psi^{(1)}_{g,h}(x)\dd \mu_g(x)<\infty\,,
\] then the wave-operators $\mathscr W_\pm(D_h,D_g,I_{g,h})$ exist and are complete. Moreover, the $\mathscr W_\pm(D_h,D_g,I_{g,h})$ are partial isometries with initial space $\Gamma_{L^2}(M,\Sigma_gM)^{\mathrm{ac}}(D_g)$ and final space $\Gamma_{L^2}(M,\Sigma_h M)^{\mathrm{ac}}(D_h)$. In particular, we have $\spec_{\mathrm{ac}}(D_g)=\spec_{\mathrm{ac}}(D_h)=\R$.\\
b) If
\[
\int_M \Psi^{(2)}_{g,h}(x)\dd \mu_g(x)<\infty\,,
\] then the wave-operators $\mathscr W_\pm(D^2_h,D^2_g,I_{g,h})$ exist and are complete. Moreover, the $\mathscr W_\pm(D^2_h,D^2_g,I_{g,h})$ are partial isometries with initial space $\Gamma_{L^2}(M,\Sigma_gM)^{\mathrm{ac}}(D^2_g)$ and final space $\Gamma_{L^2}(M,\Sigma_hM)^{\mathrm{ac}}(D_h^2)$. In particular, we have $\spec_{\mathrm{ac}}(D^2_h)=\spec_{\mathrm{ac}}(D^2_g)=[0,\infty)$.
\end{example}

\section{Application to Ricci Flow}\label{ricci}

Applying our main result to $g$ and $h$ running through a Ricci flow, we obtain the below result on the stability of the absolutely continuous spectrum, in which we adjust the notation slightly for convenience by indexing any quantity associated with a Riemannian metric $g_s$ simply by the family parameter $s$.
\begin{theorem}\label{thm:ricci-flow}
	Let $S>0$, $\kappa \in\R$ and let $(g_s)_{s\in[0,S]}$ be a smooth family of Riemannian metrics on $M$. Assume
	\begin{itemize}
		\item[(i)] $g_0$ is geodesically complete with $|R_0|_0\leq C <\infty$;
		\item[(ii)] $(g_s)_{s\in[0,S]}$ evolves under a Ricci type flow
		\[
			\frac{\partial}{\partial s}g_s = \kappa \mathrm{Ric}_s \text{ for all } s\in[0,S]\,;
		\]
		\item[(iii)] there exist positive constants $C_0,C_1$ such that
		\begin{align*}
			|R_s|_s \leq C_0 \quad \text{ and } \quad |\nabla^sR_s|_s\leq C_1/s \quad \text{ for all } s\in (0,S]\,.
		\end{align*}
	\end{itemize}

	Setting, for $s_0\in(0,S)$ and $x\in M$,
	\begin{align*}
			\mathcal A_{s_0}(x) &:= \sup\left\{|\mathrm{Ric}_s(v,v)|\,:\, v\in T_xM, |v|_s\leq 1, s\in[s_0,S] \right\}\,,\\
				\mathcal B_{s_0}(x) &:= \sup\{|\nabla^s_v\mathrm{Ric}_s(u,w) + \nabla^s_u\mathrm{Ric}_s(v,w)-\nabla^s_w\mathrm{Ric}_s(u,v)| \,:\, u,v,w\in T_xM,\\ &\mkern72mu|u|_s,|v|_s,|w|_s\leq 1, s\in[s_0,S] \}\,,
	\end{align*}
	one has:\\
	a) If for some $s_0\in (0,S)$ and $t>0$
	\[
		\int_M\! \frac{\max(\sinh\left(\tfrac n4 (S-s_0)|\kappa|\mathcal A_{s_0}(x)\right),\sinh^2\left(\tfrac n4 (S-s_0)|\kappa|\mathcal A_{s_0}(x)\right),\mathcal B^2_{s_0}(x))}{\mu_{s_0}( B_{s_0}(x,\sqrt t))}\dd\mu_{s_0}(x) < \infty
	\]
	then the wave operators $\mathscr W_\pm(D_s,D_{s_0},I_{s_0,s})$ exist and are complete and one has $\spec_{\mathrm{ac}}(D_s) = \spec_{\mathrm{ac}}(D_{s_0})$ for all $s\in [s_0,S]$.\\
	b) If for some $s_0\in (0,S)$ and $t>0$
	\begin{equation}\label{eqn:ricci-flow-00}
		\int_M\! \frac{\max\left(\sinh\left(\tfrac n4 (S-s_0)|\kappa|\mathcal A_{s_0}(x)\right),\mathcal B_{s_0}(x)\right)}{\mu_{s_0}( B_{s_0}(x,\sqrt t))}\dd\mu_{s_0}(x) < \infty
	\end{equation}
	then the wave operators $\mathscr W_\pm(D^2_s,D^2_{s_0},I_{s_0,s})$ exist and are complete and one has $\spec_{\mathrm{ac}}(D^2_s) = \spec_{\mathrm{ac}}(D^2_{s_0})$ for all $s\in [s_0,S]$.
\end{theorem}

\begin{proof}
	We will only prove b) leaving the apparent modifications for the proof of a) to the reader. The symbol $C$ will stand for a positive constant whose actual value might change from line to line but whose existence is assured by the assumptions of the theorem.
	
	Firstly, it is well known that the Ricci flow equation in conjunction with \textrm{(i)} implies that $g_s\sim g_0$ for all $s\in[0,S]$ (see, e.g., the proof of Theorem\ 1.2 in \cite{shi}). In particular, all $g_s$ are geodesically complete.
	
	Next, we restrict to a subfamily $(g_s)_{s\in[s_0,S]}$ of metrics, where $s_0\in (0,S)$, so that by assumption \textrm{(iii)} we have
	\begin{align}\label{eqn:ricci-flow-01}
		|\nabla^sR_s|_s \leq C \quad \text{ for all } s\in [s_0,S]\,,
	\end{align}
	which implies, in particular, that $\Psi_s$ is bounded.
		
	It remains to prove that $\omega_{s_0,s}$ is bounded and that the integrability \eqref{eqn:ricci-flow-00} implies
	\begin{equation}\label{eqn:ricci-flow-02.5}
		\int_M\!\Psi^{(2)}_{s_0,s}(x)\mu_{s_0}( B_{s_0}(x,\sqrt t))^{-1}\dd\mu_{s_0}(x) < \infty
	\end{equation}
	for all $s\in[s_0,S]$.
	
	We refer the reader to the proof of \cite[Corollary 2.4]{gt} for the (using Gronwall's lemma, simple to prove) fact that
	\[
		\delta_{s_0,s}(x) \leq 2\sinh(\tfrac n4 (S-s_0)|\kappa|\mathcal A_{s_0}(x))\,,
	\]
	so that
	\begin{equation}\label{eqn:ricci-flow-03}
		\max(\delta_{s_0,s}(x)\Psi_s(x),\delta_{s_0,s}(x)\Psi_{s_0}(x))	\leq C \cdot \sinh(\tfrac n4 (S-s_0)|\kappa|\mathcal A_{s_0}(x))\,.
	\end{equation}
	
	The fundamental theorem of calculus gives us (pointwise)
	\[
		\nabla^s-\nabla^{s_0} = \int_{s_0}^{s}\frac{\partial}{\partial \sigma}\nabla^\sigma\dd\sigma\,.
	\]

	Taking norms, we get
	\[
		|\nabla^s-\nabla^{s_0}|_{s_0}\leq \int_{s_0}^{s}\left|\frac{\partial}{\partial \sigma}\nabla^\sigma \right|_{s_0}\dd\sigma \leq  \int_{s_0}^{s}|\mathscr  A'^{-1}_{\sigma,s_0}||\mathscr  A_{\sigma,s_0}|^{1/2}\left|\frac{\partial}{\partial \sigma}\nabla^\sigma \right|_{\sigma}\dd\sigma\leq C\int_{s_0}^{s}\left|\frac{\partial}{\partial \sigma}\nabla^\sigma \right|_{\sigma}\dd\sigma\,,
	\]
	where the last inequality follows again from $g_{s_0}\sim g_{\sigma}$.
	
	Letting $(e_1,\ldots,e_n)$ be a $g_\sigma$-orthonormal basis of $T_xM$ and using \cite[Proposition 2.3.1]{topping}, we calculate
	\begin{align*}
	\left|\frac{\partial}{\partial \sigma}\nabla^\sigma \right|^2_{\sigma} &= \sum_{i,j,k=1}^{n}g_\sigma\left(\frac{\partial}{\partial \sigma}\nabla^\sigma_{e_i}e_j,e_k\right)^2 = \sum_{i,j,k=1}^{n} \left(\nabla^\sigma_{e_j}\mathrm{Ric}(e_i,e_k) +\nabla^\sigma_{e_i}\mathrm{Ric}(e_j,e_k)-\nabla^\sigma_{e_k}\mathrm{Ric}(e_i,e_j) \right)^2\\
	&\leq n^3 \mathcal B^2_{s_0}(x)\,,
	\end{align*}
	which implies
	\begin{equation}\label{eqn:ricci-flow-04}
		\omega_{s_0,s}(x)\leq C \cdot \mathcal B_{s_0}(x)\,.
	\end{equation}
	Note that $\omega_{s_0,s}$ is bounded since by \eqref{eqn:ricci-flow-01}
	\[
		\sup_{x\in M} \mathcal B_{s_0}(x) < \infty\,.
	\]
	
	Now \eqref{eqn:ricci-flow-02.5} follows from \eqref{eqn:ricci-flow-00}, \eqref{eqn:ricci-flow-03} and \eqref{eqn:ricci-flow-04}.
\end{proof}

\appendix
\section{Stochastic Analysis}
The material presented in this section is mostly standard and can be found, e.g., in \cite{ikeda}. \\
Let $(\Omega, \IFF, \IFF_*, \IP)$ be a filtered probability space, which means that $(\Omega, \IFF, \IP)$ is a probability space $\IFF_*=(\IFF_t)_{t\geq 0}$ is a filtration of the $\sigma$-algebra $\IFF$ by sub-$\sigma$-algebras. We will assume that $(\Omega, \IFF, \IFF_*, \IP)$ satisfies the so called usual conditions, which means that $\IFF_*$ is right continuous in the sense that 
$$
\IFF_s=\bigcap _{s<t} \IFF_t \quad\text{ for all $s\geq 0$,}
$$
and that $\IFF_0$ contains all sets $N\in\IFF$ with $\IP(N)=0$. This assumptions guarantees that one can pick continuous versions of sufficiently well-behaved processes which are defined on this probability space.\\
A \emph{stopping time} is a random variable, that is, an ($\IFF$-) measurable map
$$
\tau:\Omega\longrightarrow [0,\infty]
$$
such that $\{\tau\leq t\}\in\IFF_t$ for all $t\geq 0$. A map with values in a topological space (equipped with its Borel-$\sigma$-algebra)
$$
\mathsf{Y}:[0,\infty)\times \Omega\longrightarrow \mathscr{X}
$$
is called a \emph{process}, if for all $t\geq 0$ the map 
$$
\mathsf{Y}_t:\Omega\longrightarrow \mathscr{X}
$$
is measurable. Then the maps
$$
\mathsf{Y}(\omega):[0,\infty)\longrightarrow \mathscr{X},\quad \omega\in\Omega,
$$
are called the \emph{paths} of $\mathsf{Y}$. The process $\mathsf{Y}$ is called \emph{adapted}, if $\mathsf{Y}_t$ is $\IFF_t$-measurable for all $t\geq 0$. If $\mathsf{Y}$ adapted with continuous paths, if $\mathscr{X}$ is completely metrizable, and $U
\subset \mathscr{X}$ is an open subset, then the \emph{first exit time}  
$$
\tau^U_{\mathsf{Y}} :=\inf\{t\geq 0: \mathsf{Y}_{t}  \notin U\}: \Omega \longrightarrow [0,\infty]
$$
of $\mathsf{Y}$ from $U$ provides an example of a stopping time.\\
Assume from now on that $\mathscr{X}$ is a finite dimensional Hilbert space. Given $t\geq 0$, the orthogonal projection
$$
L^2(\Omega,\IFF,\IP; \mathscr{X} )\longrightarrow L^2(\Omega,\IFF_t,\IP;\mathscr{X})
$$
extends uniquely to a bounded linear map
$$
\pi_t:L^1(\Omega,\IFF,\IP;\mathscr{X})\longrightarrow L^1(\Omega,\IFF_t,\IP;\mathscr{X}),
$$
called the \emph{conditional expectation given $\IFF_t$}. One has 
\begin{align}\label{aasooo}
\int \pi_t f\dd \IP = \int f \dd \IP\quad\text{for all $t\geq 0$, $f\in L^1(\Omega,\IFF,\IP;\mathscr{X})$}.
\end{align}
Then a process
$$
\mathsf{Y}:[0,\infty)\times \Omega\longrightarrow \mathscr{X}
$$ 
with 
$$
\int |\mathsf{Y}_t |\dd \IP<\infty\quad\text{for all $t\geq 0$}
$$
is called a \emph{martingale}, if for all $0\leq s\leq t$ one has
\begin{align}\label{aasooo2}
\mathsf{Y}_s=\pi_s \mathsf{Y}_t,
\end{align}
and a \emph{local martingale}, if it is adapted and there exists a strictly increasing sequence of stopping times $\tau_n\nearrow\infty$ such that for all $n\in\IN$ the stopped process 
$$
\mathsf{Y}^{\tau_n}:[0,\infty)\times \Omega\longrightarrow \mathscr{X},\quad \mathsf{Y}^{\tau_n}_t:=\mathsf{Y}_{t\wedge \tau}
$$
is a martingale. Any martingale is adapted and a local martingale. A local martingale $\mathsf{Y}$ is a martingale, if it is uniformly integrable in the following sense:
$$
\int \sup_{r\in [0,t]} |\mathsf{Y}_t| \dd \IP<\infty\quad\text{for all $t>0$}.
$$
Note that if $\mathsf{Y}$ is a martingale, then in view of (\ref{aasooo}) and (\ref{aasooo2}) one has
$$
\int \mathsf{Y}_t \dd \IP= \int \pi_0\mathsf{Y}_t \dd \IP= \int \mathsf{Y}_0 \dd \IP,
$$
so that martingales have a constant expectation.\\
A process
$$
\mathsf{W}:[0,\infty)\times \Omega\longrightarrow \mathscr{X}
$$ 
with continuous paths is called an \emph{adapted Euclidean Brownian motion starting in $x_0\in \mathscr{X}$}, if it is adapted and the transition density with respect to the Lebesgue measure on $\mathscr{X}$ is given by the Gauss kernel 
$$
(0,\infty)\times \mathscr{X}\ni (t,y) \longmapsto e^{-t\Delta}(x_0,y)=(4\pi t)^{-\frac{\dim(\mathscr{X})}{2}}e^{-\frac{|x_0-y|^2}{4t}}\in (0,\infty), 
$$
in the sense that for any $m\in\IN$, any finite sequence of times $0<t_1< \dots< t_m$ and all Borel sets $A_1,\dots,A_m\subset \mathscr{X}$, setting $\delta_j:=t_{j+1}-t_j$ with $t_0:=0$, one has
\begin{align*}
\IP\{\mathsf{W}_{t_1}(x_0)\in A_1,\dots,\mathsf{W}_{t_m}(x_0)\in A_m\} =\int\cdots \int 1_{A_1}(x_1)e^{-\delta_0\Delta}(x_0,x_1) \cdots 
 1_{A_m}(x_m) e^{-\delta_{m-1}\Delta}(x_{m-1},x_m) \dd x_1\cdots \dd x_m.
\end{align*}
Above, $\Delta\geq 0$ denotes (unique self-adoint realization of) the Euclidean Laplace-Operator in $L^2(\mathscr{X})$, and as the notation indicates, the Gauss kernel is the integral kernel of the associated heat semigroup. This observation allows to define Brownian motion on smooth Riemannian manifolds, too. \\
Any adapted Euclidean Brownian motion $\mathsf{W}$ turns out to be a martingale with
$$
\int |\mathsf{W}_t |^2\dd \IP<\infty\quad\text{for all $t\geq 0$}.
$$
Assume now $\mathscr{X}_1$, $\mathscr{X}_2$, $\mathscr{X}_3$ are finite dimensional Hilbert spaces, 
$$
\mathsf{W}:[0,\infty)\times \Omega\longrightarrow \mathscr{X}_1
$$
is an adapted Euclidean Brownian motion, 
$$
\mathsf{Y}:[0,\infty)\times \Omega\longrightarrow \mathscr{X}_2
$$
is an adapted process with continuous paths, and 
$$
\mathscr{G}\in \Hom\big(\mathscr{X}_1,\Hom(\mathscr{X}_2,\mathscr{X}_3)\big).
$$
Then
$$
\int^{\bullet}_0  \mathscr{G}(\dd \mathsf{W}_s) \mathsf{Y}_s: [0,\infty)\times \Omega\longrightarrow \mathscr{X}_3
$$
denotes the \emph{Ito stochastic integral}, which can be defined as the uniquely given local martingale with continuous paths starting in $0$ which for all $t\geq 0$ satisfies
\begin{align}\label{spast}
\int^{t}_0  \mathscr{G}(\dd \mathsf{W}_s) \mathsf{Y}_s=\underset{n\to\infty}{\mathrm{l.i.p.}}  \      \sum^n_{k=1} \mathscr{G}\big(W_{\frac{kt}{n}}-W_{\frac{(k-1)t}{n}}\big) Y_{\frac{(k-1)t}{n}},
\end{align}
where $\mathrm{l.i.p.}$ stands for the limit in probability, that is.
$$
\underset{n\to\infty}{\mathrm{l.i.p.}} \ X_n=X\quad\text{if and only if for all $\epsilon>0$ one has $\lim_{n\to\infty}P(|X_n-X|>\epsilon)=0$.}
$$
In (\ref{spast}) it is not important to take the uniform partition of $[0,t]$ (as we did), in the sense that one can take any sequence of partitions whose mesh tends to zero. It is important, however, to  evaluate $Y$ in the left points of the underlying sequence of partitions. For example, taking $\frac{1}{2}(Y_{\frac{kt}{n}}-Y_{\frac{(k-1)t}{n}})$ instead of $Y_{\frac{(k-1)t}{n}}$ would lead to the so called \emph{Stratonovic stochastic integral}.\vspace{1mm}

In the above situation, the \emph{Burkholder-Davis-Gundy inequality} states that for all $p\in (0,\infty)$ there exists a constant $C_p<\infty$ such that for all finite stopping times $\tau$ one has 
\begin{align}\label{bdg}
\int \left|\sup_{r\in [0,\tau]}\int^{r}_0  \mathscr{G} (\dd\mathsf{W}_s) Y_s\right|^p  \dd \IP\leq C_p \int \left(\int^{\tau}_0|Y_s|^2 \dd s \right)^{p/2}\dd \IP\in [0,\infty].
\end{align}
This estimate follows, e.g., from Theorem 1 in \cite{ren2008}.\\
Finally we remark that, with obvious modifications, all of the above definitions and results carry over the case where one replaces $[0,\infty)$ with a time horizon of the form $[0,T]$, where $T>0$.

\section{The Belopol\rq{}skii-Birman Theorem}\label{saab}

We recall that given a self-adjoint operator $H$ in a complex Hilbert space $\IHH$, the closed subspace of $H$-absolutely continuous states is defined as
$$
\IHH^{\mathrm{ac}}(H):=\{\psi\in\IHH: \left\| E_H(\cdot)\psi\right \|^2\text{ is absolutely continuous w.r.t. Lebesgue measure} \}\subseteq \IHH,
$$
where $E_H$ denotes the spectral measure of $H$. Then $\IHH^{\mathrm{ac}}(H) $ reduces $H$ and the spectrum of the induced self-adjoint operator in $\IHH^{\mathrm{ac}}(H)$ is denoted by $\mathrm{Spec}_{\mathrm{ac}}(H)$ and called the \emph{absolutely continuous spectrum of $H$.} We denote with 
$$
\pi_{\mathrm{ac}}(H): \IHH\longrightarrow \IHH^{\mathrm{ac}}(H)
$$ 
the orthogonal projection.

\begin{theorem}\label{belo}\emph{(Belopol\rq{}skii-Birman)} For $k=1,2$, let $H_k$ be a self-adjoint operator in a complex Hilbert space $\IHH_k$, where $\pi_{\mathrm{ac}}(H_k)$ denotes the projection onto the $H_k$-absolutely continuous subspace of $\IHH_k$. Assume that $I\in\ILL(\IHH_1, \IHH_2)$ is a bounded operator which has a two-sided bounded inverse.\\
a) Assume that $H_k\geq 0$ and that
\begin{itemize}
\item one has $I\dom(\sqrt{H_1})=\dom(\sqrt{H_2})$,
\item the operator
\begin{align*}
(I^*I-1)\exp(-sH_1):\IHH_1\to\IHH_1 \>\>\text{ is Hilbert-Schmidt (or more generally: compact) for some $s>0$},
\end{align*}
\item there exists a trace class operator $\mathscr{T}:\IHH_1\to \IHH_2$ and a number $s>0$ such that for all $f_2\in\dom(H_2)$, $f_1\in\dom(H_1)$ one has
\begin{align*}
\left\langle f_2 ,\mathscr{T}f_1\right\rangle=\left\langle H_2f_2, \exp(-sH_{2}) I \exp(-sH_{1})f_1\right\rangle -\left\langle f_2, \exp(-sH_{2}) I \exp(-sH_{1}) H_1f_1\right\rangle. 
\end{align*}
\end{itemize}
Then the wave operators 
$$
\mathscr{W}_{\pm}(H_{2},H_1, I)=\slim_{t\to\pm\infty}\exp(itH_{2})I\exp(-itH_{1})\pi_{\mathrm{ac}}(H_1)
$$
exist\footnote{$\slim_{t\to\pm\infty}$ stands for the strong limit.} and are complete, where completeness means that 
$$
\left(\mathrm{Ker} \: \mathscr{W}_{\pm}(H_{2},H_1, I)\right)^{\perp}=\IHH^{\mathrm{ac}}_1(H_1), \quad\overline{\mathrm{Ran} \: \mathscr{W}_{\pm}(H_{2},H_1, I)}=\IHH^{\mathrm{ac}}_2(H_2).
$$
Moreover, $\mathscr{W}_{\pm}\big(H_{2},H_1, I\big)$ are partial isometries with inital space $\IHH^{\mathrm{ac}}_1(H_1)$ and final space $\IHH^{\mathrm{ac}}_2(H_2)$, and one has $\mathrm{Spec}_{\mathrm{ac}}(H_1)=\mathrm{Spec}_{\mathrm{ac}}(H_2)$.\vspace{1mm}

b) Assume that 
\begin{itemize}
\item one has $I\dom(H_1)=\dom(H_2)$,
\item the operator
\begin{align*}
(I^*I-1)\exp(-sH_1^2):\IHH_1\to\IHH_1 \>\>\text{ is Hilbert-Schmidt (or more generally: compact) for some $s>0$},
\end{align*}
\item there exists a trace class operator $\mathscr{R}:\IHH_1\to \IHH_2$ and a number $s>0$ such that for all $f_2\in\dom(H_2)$, $f_1\in\dom(H_1)$ one has
\begin{align*}
\left\langle f_2 ,\mathscr{R}f_1\right\rangle =\left\langle H_2f_2, \exp(-sH_{2}^2) I \exp(-sH_{1}^2)f_1\right\rangle  -\left\langle f_2, \exp(-sH_{2}^2) I \exp(-sH_{1}^2) H_1f_1\right\rangle . 
\end{align*}
\end{itemize}
Then the same conclusions as in a) holds.
\end{theorem}

\begin{proof} See the appendix of \cite{gt} and the references therein.
\end{proof}



\end{document}